\documentclass{amsart}
\usepackage{amsmath}
\usepackage{amscd}
\usepackage{amssymb}
\usepackage{amsthm}
\usepackage{enumerate}
\usepackage[normalem]{ulem}
\usepackage{mathtools}
\usepackage[usenames,dvipsnames]{xcolor}
\usepackage{stmaryrd}
\usepackage[left=1in,top=1in,right=1in]{geometry}
\usepackage[T2A]{fontenc}
\usepackage[utf8]{inputenc}
\usepackage{mathrsfs}

\newtheorem{thm}{Theorem}[section]
\newtheorem{lemm}[thm]{Lemma}
\newtheorem{prop}[thm]{Proposition}
\newtheorem{cor}[thm]{Corollary}

\theoremstyle{definition}
\newtheorem{defn}[thm]{Definition}

\newtheorem{rem}[thm]{Remark}

\numberwithin{equation}{section}

\newcommand{\bC}{{\mathbb C}}

\newcommand{\bF}{{\mathbb F}}

\newcommand{\bM}{{\mathbb M}}
\newcommand{\bN}{{\mathbb N}}

\newcommand{\bR}{{\mathbb R}}

\newcommand{\cM}{{\mathcal M}}

\newcommand{\cO}{{\mathcal O}}

\renewcommand{\Re}{\operatorname{Re}}
\renewcommand{\Im}{\operatorname{Im}}

\DeclareMathOperator{\id}{id}

\DeclareMathOperator{\tr}{tr}

\DeclareMathOperator{\Tr}{Tr}

\DeclareMathOperator{\Span}{span}

\DeclareMathOperator{\ad}{ad}

\DeclareMathOperator{\orb}{orb}
\DeclareMathOperator{\sa}{sa}

\DeclarePairedDelimiter{\norm}{\lVert}{\rVert}
\DeclarePairedDelimiter{\ang}{\langle}{\rangle}

\begin{document}
	
	\title{Property (T) and strong 1-boundedness for von Neumann algebras}
	
	\author{Ben Hayes}
	\address{\parbox{\linewidth}{Department of Mathematics, University of Virginia, \\
			141 Cabell Drive, Kerchof Hall,
			P.O. Box 400137
			Charlottesville, VA 22904}}
	\email{brh5c@virginia.edu}
	\urladdr{https://sites.google.com/site/benhayeshomepage/home}
	
	\author{David Jekel}
	\address{\parbox{\linewidth}{Department of Mathematics, University of California, \\
			San Diego, 9500 Gilman Drive \# 0112, La Jolla, CA 92093}}
	\email{djekel@ucsd.edu}
	\urladdr{http://davidjekel.com}
	
	\author{Srivatsav Kunnawalkam Elayavalli}
	\address{\parbox{\linewidth}{Department of Mathematics, Vanderbilt University, \\
			1326 Stevenson Center, Station B 407807, Nashville, TN 37240}}
	\email{srivatsav.kunnawalkam.elayavalli@vanderbilt.edu}
	\urladdr{https://sites.google.com/view/srivatsavke}
	
	\thanks{B. Hayes gratefully acknowledges support from the NSF grant DMS-2000105.  D. Jekel was supported by NSF grant DMS-2002826.}
	
	\begin{abstract}
		The notion of strong 1-boundedness for  finite von Neumann algebras was introduced in \cite{Jung2007}. This framework provided a free probabilistic approach to study rigidity properties and classification of finite von Neumann algebras. In this paper, we prove that tracial von Neumann algebras with a finite Kazhdan set are strongly 1-bounded. This includes  all Property (T) von Neumann algebras with finite dimensional center and group von Neumann algebras of Property (T) groups. This result generalizes all the previous results in this direction due to Voiculescu, Ge, Ge-Shen, Connes-Shlyakhtenko, Jung-Shlyakhtenko, Jung, and Shlyakhtenko. Our proofs are based on analysis of covering estimates of microstate spaces using an iteration technique in the spirit of Jung.
	\end{abstract}
	
	\maketitle
	
	
	\section{Introduction}
	
	\subsection{Main results}
	
	A \emph{tracial von Neumann algebra} is a pair $(M,\tau)$ of a finite von Neumann algebra and a faithful normal tracial state.  Tracial von Neumann algebras have long been seen as a non-commutative analog of  probability spaces.  Inspired by Bolzmann's formulation of entropy, Voiculescu \cite{VoiculescuFreeEntropy2, Voiculescu1996} introduced the microstates free entropy and microstates free entropy dimension $\delta_{0}(x)$ associated to a tuple $x = (x_{1},\cdots,x_{k})$ of self-adjoint elements in a tracial von Neumann algebra, which measure the quantity of matricial approximations that $x$ has.  Microstates free entropy dimension was used by Voiculescu \cite{Voiculescu1996} and Ge \cite{GePrime} respectively to show that (interpolated) free group factors have no Cartan subalgebras and are prime, and since then, free entropy techniques have had many other applications to the structural properties of interpolated free group factors.
	
	However, one limitation of microstates free entropy dimension is that $\delta_0(x)$ is not known to be an invariant of the tracial von Neumann algebra generated by $x$. Jung \cite{Jung2007} offered a remedy by defining \emph{strong $1$-boundedness} of a finite tuple $x$ in a tracial von Neumann algebra $(M,\tau)$.  Strong $1$-boundedness of $x$ implies that is $x$ has free entropy dimension at most $1$, but unlike the case with free entropy dimension, we know that if two tuples $x$ and $y$ generate the same von Neumann algebra, then strong $1$-boundedness of $x$ is equivalent to strong $1$-boundedness of $y$ \cite[Theorem 3.2]{Jung2007}, hence strong $1$-boundedness is an invariant of a finitely generated tracial von Neumann algebra.  Implicit in Jung's work, and explicitly given by the first named author in \cite{Hayes2018}, is the notion of the $1$-bounded entropy $h(x)$ of a tuple $x$; strong $1$-boundedness of $x$ is equivalent to $h(x) < \infty$. Moreover, $h(x)$ only depends on the tracial von Neumann algebra generated by $x$, and thus $h(M)$ is well-defined for a finitely generated tracial von Neumann algebra $(M,\tau)$ \cite[Theorem A.9]{Hayes2018} (these definitions and results also extend to infinite tuples, see \cite[Definition A.2]{Hayes2018}).
	
	The $1$-bounded entropy behaves nicely under many natural operations with von Neumann algebras, which leads to free entropy proofs of various indecomposability results for free product von Neumann algebras (and other similar algebras), see Section \ref{sec:S1B background}.  For example, $h$ has the following subadditivity property:  If $P$ and $Q$ are von Neumann subalgebras of $(M,\tau)$ and $P \cap Q$ is diffuse, then $h(P \vee Q) \leq h(P) + h(Q)$.  Hence, a non-strongly $1$-bounded von Neumann algebra $M$ can never be generated by two strongly $1$-bounded subalgebras $P$ and $Q$ that intersect diffusely \cite[Corollary 4.2]{Jung2007} \cite[Lemma A.12]{Hayes2018}.  Hence, each time we find a new class of von Neumann algebras that is strongly $1$-bounded or not strongly $1$-bounded, we expand the applications of free entropy dimension theory.  For instance, strongly $1$-bounded von Neumann algebras include those that are amenable, have property Gamma, have a Cartan subalgebra, or even have a quasi-regular hyperfinite subalgebra.  The main known examples of non-strongly-$1$-bounded von Neumann algebras are free products (or free with amalgamation over a totally atomic subalgebra) of Connes-embeddable tracial von Neumann algebras, including in particular the interpolated free group factors, which are a very important but not well understood family of von Neumann algebras. Because of the permanence properties of $1$-bounded entropy given in Section \ref{sec:S1B background}, $1$-bounded entropy can be used to prove structural results about free products, some of which are still not accessible by other methods; see for instance \cite[Corollary 4.2]{Jung2007}, \cite[Theorem 1.3, Corollary 1.7, Theorem 1.8, Corollary 1.10]{Hayes2018}.
	
	While being strongly $1$-bounded guarantees having microstates free entropy dimension at most $1$ with respect to every generating set
	(we say that $M$ has microstate free entropy dimension at most $1$ if this holds), there is a significant difference between these two concepts. Having microstates free entropy dimension at most $1$ is not known to satisfy the many permanence properties (closure under joins with diffuse intersection, normalizers, quasi-normalizers etc) that being strongly $1$-bounded enjoys; see Section \ref{sec:S1B background}. For example, it is unknown if $L(\bF_{2})$ can be generated by two algebras with microstates free entropy dimension at most $1$ with diffuse intersection (and similarly with other indecomposability results for free group factors give in \cite{Jung2007, Hayes2018}). Thus being strongly $1$-bounded is a serious improvement over having microstates free entropy dimension at most $1$.

	
	This paper undertakes a comprehensive study of the relationship between strong $1$-boundedness and Property (T).  Property (T) is a strong rigidity property introduced first in the group context by Kazhdan \cite{KazhdanTDef}. It has numerous applications to geometric and measured group theory \cite{MargulisNST, GabICM, FurmanMeasRigid, PopaCorr},  ergodic theory \cite{CWPropT, SchmidtCohom, SchmidtSpectralGap, GlasnerWeissT, PopaCoc, PopaStrongRigidity, PopaStrongRigidtyII, FurmanOERigid, IoanaCocycleSuperRIgid},  probability \cite{LyonsSchramm, GabHarmonic}, and the existence of expander graphs \cite{MargulisExpanders}. See \cite{BHV} for an extensive background on Property (T) groups. A generalization to von Neumann algebras was first defined by Connes \cite{ConnesTDef,ConnesCountable} and further developed by Connes-Jones \cite{ConnesJones}, Popa \cite{PopaL2Betti}, Popa-Peterson \cite{PPRelativeT}, and Peterson \cite{Peterson2009}. It has been used to great effect in the theory of von Neumann algebras, especially in Popa's deformation/rigidity theory,  and is a crucial component of various general results about the structure of $\textrm{II}_{1}$-factors \cite{PopaCorr, OzawaUnivSep, PopaL2Betti} many of which parallel or extend the results obtained from free entropy dimension theory; see \ref{subsec: property T discussion} for further discussion.
	
	Several previous authors gave estimates on microstates free entropy dimension for Property (T) groups and von Neumann algebras. First, Voiculescu \cite{VoiculescuPropT} showed that the standard generators of $SL_3(\mathbb{Z})$ have free entropy dimension 1, by using the sequential commutation property. Ge and Shen in \cite{GeShenPropT} generalized this result to all generators of $SL_3(\mathbb{Z})$, thus establishing that free entropy dimension is a von Neumann algebra invariant in this case.  
	The next breakthrough by Jung and Shlyakhtenko \cite{JungS} showed that a finitely generated Property (T) von Neumann algebra has free entropy dimension at most $1$ with respect to any finite generating set.
	Jung \cite[Theorem 6.9]{JungL2B},  showed that for sofic groups certain relations guarantee strong $1$-boundedness, and applied this to deduce strong $1$-boundedness for torsion-free sofic groups with two generators \cite[Section 7]{JungL2B} which are not isomorphic to the free group on two generators. Shlyakhtenko \cite{Shl2015} generalized Jung's result via a non-microstates approach and showed that all finitely presented sofic groups with vanishing first $\ell^2$ Betti number are strongly $1$-bounded. This includes finitely presented sofic Property (T) groups by the Delorme-Guichardet theorem \cite{DelormeT, GuichardetT} (see also \cite[Section 2.12]{BHV}).
	
	Our main result is a complete generalization of all the above mentioned previous results in the Property (T) setting. We show strong $1$-boundedness assuming only that $(M,\tau)$ has a finite \emph{Kazhdan tuple}, that is, some $x = (x_1,\dots,x_d) \in M^d$ such that for some constant $\gamma > 0$, for every Hilbert $M$-$M$-bimodule $H$, we have
	\[
	\norm{\xi - P_{\operatorname{central}}\xi} \leq \frac{1}{\gamma} \left( \sum_{j=1}^d \norm{x_j \xi - \xi x_j}_2^2 \right)^{1/2},
	\]
	where $P_{\operatorname{central}}: H \to H$ denotes the projection onto the subspace of central vectors.
	
	\begin{thm} \label{thm: Property T}
		If $(M,\tau)$ is a tracial von Neumann algebra that admits a Kazhdan tuple $x \in M_{\sa}^d$, then $M$ is strongly $1$-bounded.  In particular, this holds if $M = W^{*}(\pi(G))$ where $G$ is a Property (T) group and  $\pi\colon G\to \mathcal{U}(\mathcal{H})$ is a  unitary representation given by a character on $G$,\footnote{A \emph{character on $G$} is a conjugation-invariant positive-definite function $\chi: G \to \bC$.  The unitary representations of $G$ corresponding to characters are exactly those which generate a finite von Neumann algebra.} or if $M$ is a Property (T) tracial von Neumann algebra with finite-dimensional center. This also holds if $M=W^{*}(\pi(G))$ where $G$ is a Property (T) group, and $\pi\colon G\to \mathcal{U}(\mathcal{H})$ is a projective representation \footnote{This means $\pi(g)\pi(h)=c(g,h)\pi(gh)$ for all $g,h\in G$ and some $c\colon G\times G\to S^{1}$.}, provided that $W^{*}(\pi(G))$ is a finite von Neumann algebra.
	\end{thm}
	
	As mentioned earlier in the introduction, this result expands the applications of free entropy to the structure theory of finite von Neumann algebras. From the subadditivity property of $h$, we see the following: A non-strongly $1$-bounded von Neumann algebra $M$ can never be generated by a Property (T) subfactor $P$ and  another strongly $1$-bounded von Neumann subalgebra $Q$ (e.g. amenable, Property Gamma, non-prime, having a diffuse hyperfinite quasi regular subalgebra) such that $P \cap Q$ is diffuse. We remark that in the special case when $M$ is a non trivial free product it is possible to obtain some of these indecomposability results by combining existing techniques in Popa's deformation-rigidity theory. However, the advantage of the 1-bounded entropy approach is in the uniformity of being able to handle all the various cases at once.
	
	
	Our proof of Theorem \ref{thm: Property T} works directly with covering numbers of Voiculescu's matricial microstate space. It draws upon the dimension-reduction technique of Jung and Shlyakhtenko \cite{JungS} as well as the iterative technique that Jung used in his study of torsion-free sofic groups with two generators \cite{JungL2B}. The significant technical novelty in our proof is found in a careful control of the packing number estimates in an iterative fashion while using the Kazhdan constant at each step.
	
	It is natural to ask whether Theorem \ref{thm: Property T} extends to arbitrary tracial von Neumann algebras with Property (T). 
	From our analysis, we observe that the only remaining case is that of a tracial von Neumann algebra $(M,\tau)$ with Property (T) and with atomic, infinite-dimensional center (i.e. $M$ is an infinite direct sum of Property (T) factors). By a careful analysis of this case, we show the statement that all Property (T) tracial von Neuman algebras are strongly $1$-bounded is equivalent to the statement that all $\mathrm{II}_{1}$-factors with Property (T) have nonpositive $1$-bounded entropy. 
	
	\begin{prop} \label{prop: dichotomy}
		The following are equivalent:
		\begin{enumerate}[(i)]
			\item Every tracial von Neumann algebra with Property (T) is strongly $1$-bounded.
			\item Every tracial von Neumann algebra $M$ with Property (T) satisfies $h(M) \leq 0$.
			\item Every $\mathrm{II}_1$-factor $M$ with Property (T) satisfies $h(M) \leq 0$.
		\end{enumerate}
		Moreover, if $M$ is a finite von Neumann algebra with Property (T), then there exists some faithful normal tracial state $\tau$ on $M$ such that $(M,\tau)$ is strongly $1$-bounded.
	\end{prop}
	
	This problem of whether every $\mathrm{II}_1$ factor with Property (T) satisfies $h(M) \leq 0$ does not seem to accessible by current techniques since we do not even know whether there exists a tracial von Neumann algebras with $0 < h(M) < \infty$. 
	It is thus likely that Theorem \ref{thm: Property T} is the optimal result, see Sections \ref{subsec: property T discussion}, \ref{sec:amp it up} for a more detailed discussion.
	
	The proof of Proposition \ref{prop: dichotomy} is based on the behavior of $1$-bounded entropy under direct sums and matrix amplifications/compressions.  In particular, it uses the following
	formula analogous to Schreier's formula for the rank of subgroups of free groups.
	
	\begin{prop} \label{prop: amplification}
		Suppose that $M$ is a $\mathrm{II}_{1}$-factor and $\tau$ is its canonical trace. For $t \in (0,\infty)$, let $M^{t}$ be the amplification of $M$ by $t$. Then
		\[
		h(M^{t}) = \frac{1}{t^{2}}h(M).
		\]
	\end{prop}
	
	This formula should be compared with the compression formula for free group factors due to Voiculescu \cite[Theorem 3.3]{VoiculescuIFGF}, Dykema \cite{DykmaIFGF}, and Radulescu \cite{RadulescuIFGF},
	the Connes-Shlyakhtenko formula for $L^{2}$-Betti numbers \cite[Theorem 2.4]{ConnesShl},  the similar results for $L^{2}$-Betti numbers and cost of equivalence relations due to Gaboriau \cite[Proposition II.1.6]{Gab1}, \cite[Th\'{e}or\`{e}me 5.3]{Gab2},
	and Jung's results on the behavior of free entropy dimension under compression \cite{JungSubf}.
	Moreover, it has the following consequence.
	
	\begin{cor} \label{cor: fundamental group}
		Let $M$ be a $\mathrm{II}_1$ factor and let $\mathcal{F}(M) = \{t > 0: M^t \cong M\}$ be its fundamental group.  If $0 < h(M) < \infty$, then $\mathcal{F}(M) = \{1\}$.  Hence, if $M$ is a $\mathrm{II}_1$ factor with Property (T), then $h(M) \leq 0$ or $\mathcal{F}(M) = \{1\}$ (possibly both).
	\end{cor}
	
	Computation of the fundamental group is an important and challenging problem in the theory of $\mathrm{II}_1$ factors.  Connes' result \cite{ConnesCountable} that a $\mathrm{II}_1$ factor with Property (T) has countable fundamental group is considered the first milestone in the study of rigidity of $\mathrm{II}_1$ factors. Popa achieved two major breakthroughs in this area; he gave the first example of a $\mathrm{II}_1$-factor with trivial fundamental group \cite{PopaL2Betti} and then showed in  \cite{PopaStrongRigidity} that every countable subgroup of the positive real numbers can be realized as the fundamental group of some $\mathrm{II}_1$ factor. Popa conjectured that every $\mathrm{II}_1$ factor with property (T) has trivial fundamental group.  The first examples of $\mathrm{II}_1$ factors with Property (T) and trivial fundamental group were obtained only recently in \cite{PropTtrivialFG} using small cancellation techniques from geometric group theory.  As a consequence of Theorem \ref{thm: Property T} and Corollary \ref{cor: fundamental group}, if there exists a $\mathrm{II}_1$ factor $M$ with Property (T) and $h(M) > 0$, then $M$ must also satisfy $\mathcal{F}(M) = 1$ and hence would be another positive example of Popa's conjecture.  But if there does not exist a Property (T) factor $M$ with $h(M)>0$ , then all tracial von Neumann algebras with Property (T) are strongly $1$-bounded by Proposition \ref{prop: dichotomy}.
	
	\subsection{Discussion of Theorem \ref{thm: Property T}}  \label{subsec: property T discussion}
	
	Part of our motivation for Theorem \ref{thm: Property T} was to expand on the connections between free entropy dimension results and Popa's deformation/rigidity theory, in which Property (T) plays a central role \cite{PopaL2Betti, PopaCoc, PopaStrongRigidity, PopaStrongRigidtyII, AdrianBern, IoanaCocycleSuperRIgid}.
	There has been a dynamic interchange of ideas and results between deformation/rigidity theory and free probability. For example, Voiculescu's theorem on absence of Cartan subalgebras for free group factors \cite{Voiculescu1996}, and Ge's theorem on primeness of free group factors \cite{GePrime} are paralleled by Ozawa's solidity theorem \cite{OzawaSolidActa} and Ozawa and Popa's results \cite{OzPopaCartan, OzPopaII} on strong solidity of free group factors, as well as uniqueness of Cartan subalgebras for the group measure space construction of profinite actions of free groups (see \cite{PopaVaesFree,PopaVaesHyp} for the optimal results in this direction). Additionally, Jung's work on strongly $1$-bounded algebras \cite{Jung2007} motivates Peterson's result on primeness for von Neumann algebras of groups with positive first $\ell^{2}$-Betti number \cite{PetersonDeriva}. On the other hand, the numerous works in deformation/rigidity theory on free group factors or related algebras, particularly about the structure of normalizers of subalgebras,  such as in \cite{OzPopaCartan,  HoudShlStrongSolid, PetersonDeriva} provided inspiration for the malnormality results present in \cite{Hayes2018}.
	While free probability cannot currently be used to deduce uniqueness of Cartan results, or prove theorems about general hyperbolic groups or groups not Connes-embeddable, some malnormality results such as those in \cite{Hayes2018} cannot as of yet be shown using deformation/rigidity theory.
	
	One can draw a more precise parallel between the role of Property (T) in deformation/rigidity theory and the role of amenability in free entropy dimension theory as follows.  Property (T) von Neumann algebras are a canonical class to work with in deformation/rigidity because they are characterized by rigidity--a von Neumann algebra has Property (T) if and only if it is rigid with respect to every deformation inside a larger algebra \cite[Proposition 4.1]{PopaL2Betti}.  Meanwhile, amenable tracial von Neumann algebras are a canonical class to work with in free entropy dimension theory because by \cite{Connes, JungTubularity} a separable Connes-embeddable tracial von Neumann algebra is amenable if and only if all embeddings into an ultraproduct of matrix algebras are unitarily conjugate (essentially all microstates are approximately unitarily conjugate), which implies that amenable algebras have $1$-bounded entropy zero.  Hence, whereas Property (T) algebras are those which are automatically rigid, amenable algebras are those which are automatically trivial in terms of matricial microstates.  In this regard, compare the roles of Property (T) in \cite[Theorem 1.5]{Pineapple} and amenability in \cite[Corollary 1.6]{Hayes2018}.
	
	However, there is also a natural connection between Property (T) and matricial microstates.  While amenability implies that the microstate spaces are \emph{trivial} up to approximate unitary conjugation, Property (T) implies that they are \emph{discrete} up to approximate unitary conjugation and removal of a small corner, as Jung and Shlyakhtenko realized in \cite{JungS}.  To make a direct connection between Property (T) and microstate spaces, one turns a sequence of microstates into an ultraproduct embedding and uses Property (T) to show that two embeddings that are close on a generating set have large corners that are unitarily conjugate (see Lemma \ref{lemm: property T microstates}); the argument is a typical example of Popa's intertwining by bimodules technique, an important tool in deformation/rigidity theory.  However, applying this estimate naively only results in bounding the free entropy dimension by $1$, and so the problem of strong $1$-boundedness for general Property (T) factors and group von Neumann algebras posed in \cite[Remark 2.4]{JungS} remained open.  Our proof uses the Kazhdan tuple to get more explicit control over the $\eta$-covering numbers of microstate spaces in terms of $\varepsilon$-covering numbers for $\eta \leq \varepsilon$, and then rather than immediately taking $\eta \to 0$, we iteratively estimate the covering numbers for smaller and smaller $\varepsilon$ as in \cite{JungL2B}.
	
	We show in \S \ref{sec:FED d sums} that Theorem \ref{thm: Property T} recovers Jung-Shlyakhtenko's result that any finitely generated Property (T) algebra has microstates free entropy dimension at most $1$ with respect to any finite generating set. 
	
	Another natural question is whether Theorem \ref{thm: Property T} generalizes to all Property (T) von Neumann algebras, rather than only those with finite-dimensional center.  Of course, if a von Neumann algebra has diffuse center, then its $1$-bounded entropy is automatically less than or equal to zero.  Thus, the remaining case is a Property (T) von Neumann algebra that is a countable direct sum of factors.  By \cite[Proposition 4.7]{PopaL2Betti}, a direct sum of tracial von Neumann algebras has Property (T) if and only if each direct summand has Property (T).  We lower-bound the $1$-bounded entropy in terms of the $1$-bounded entropy of the direct summands.  Using this, we can prove Proposition \ref{prop: dichotomy}.
	Hence, strong $1$-boundedness of general Property (T) von Neumann algebras is as difficult as showing $1$-bounded entropy less than or equal to zero for Property (T) factors, which does not seem accessible by current techniques (if it is even true).
	
	Finally, we remark that another possible route to deduce strong $1$-boundedness from Property (T) would be to go through the arguments with $\ell^2$-Betti numbers as in \cite{Shl2015}.  Shlyakhtenko showed that every finitely presented sofic group with vanishing first $\ell^2$ Betti number is strongly $1$-bounded, which in particular includes finitely presented Property (T) sofic groups.  However, an essential ingredient in this approach is to verify the positivity of a certain Fuglede-Kadison determinant, and in the group case the only known method to do this is through soficity \cite{ElekSzaboDeterminant}.  Although Shlyakhtenko's methods have been generalized to $*$-algebras that are not group algebras \cite{BrannanVergnioux}, this still requires some way of controlling the Fuglede-Kadison determinant, hence it is currently not possible to adapt this method to general Property (T) von Neumann algebras (or even general Property (T) groups) without some analog of soficity.
	
	\subsection{Organization of the paper}
	
	We close by discussing the organization of the paper. We start in Section \ref{sec:preliminary} by giving some preliminary definitions on tracial von Neumann algebras, noncommutative laws, and the $1$-bounded entropy. This includes a discussion of permanence properties of the $1$-bounded entropy and a list of examples of algebras with nonpositive $1$-bounded entropy. In Section  \ref{sec: Property (T)} we prove Theorem \ref{thm: Property T} that Property (T) algebras with a finite Kazhdan set are strongly $1$-bounded. Section \ref{sec:d sums and matrix amp} studies the behavior of $h$ under direct sums and matrix amplification, proves Proposition \ref{prop: dichotomy}, and shows how to recover Jung and Shlyakhtenko's result that Property (T) von Neumann algebras have free entropy dimension at most $1$ from Theorem \ref{thm: Property T}.
	
	\subsection*{Acknowledgements}
	
	Dima Shlyakhtenko deserves special thanks for inspiring us to work on this problem.  We thank Thomas Sinclair and Ionut Chifan for helpful discussions about the applications of the paper, and  Dima Shlyakthenko and Adrian Ioana for comments on a draft of the paper.
	
	\section{Background}\label{sec:preliminary}
	
	\subsection{Tracial von Neumann algebras and non-commutative laws}
	
	A tracial von Neumann algebra is a pair $(M,\tau)$ where $M$ is a von Neumann algebra and $\tau\colon M\to \bC$ is a faithful, normal, tracial state. We will primarily be interested in cases where $M$ is \emph{diffuse}, i.e. it has no nonzero minimal projections. An interesting class of tracial von Neumann algebras are the \emph{group von Neumann algebras}. Given a  discrete group $G$ the \emph{left regular representation} $\lambda\colon G\to \mathcal{U}(\ell^{2}(G))$ is given by
	\[(\lambda(g)\xi)(h)=\xi(g^{-1}h)\mbox{ for $\xi\in\ell^{2}(G)$,$g,h\in G$.}\]
	The group von Neumann algebra $L(G)$ is defined to be \[
	\overline{\Span\{\lambda(g):g\in G\}}^{SOT}.
	\]
	The group von Neumann algebra can be turned into a tracial von Neumann algebra by defining $\tau\colon L(G)\to \bC$ by $\tau(x)=\ang{x(\delta_{1}),\delta_{1}}$. We may also view $\mathbb{M}_n(\bC)$ as a tracial von Neumann algebra with the tracial state $\tr_n$ given by
	\[
	\tr_{n}(A) = \frac{1}{n}\sum_{i=1}^{n}A_{ii}.
	\]
	The group von Neumann algebra is diffuse if and only if $G$ is infinite. For a von Neumann algebra $M$, we use $M_{\sa}$ for the self-adjoint elements of $M,$ and $\mathcal{U}(M)$ for the unitary elements of $M$.
	
	Since  tracial von Neumann algebras $(M,\tau)$ with $M$ abelian correspond precisely to probability spaces, we may think of tracial von Neumann algebras as (a special case of) \emph{noncommutative probability spaces}. Following this intuition, given a tracial von Neumann algebra $(M,\tau)$ and  $1\leq p <\infty$, we define the $\|\cdot\|_{p}$ on $M$ by
	\[
	\|x\|_{p}=\tau(|x|^{p})^{1/p},\mbox{ where $|x|=(x^{*}x)^{1/2}$.}
	\]
	It can be shown \cite{Dixmier1953} that this is indeed a norm on $M$.  We use the notation $\norm{x}_\infty$ for the operator norm.  Moreover, the definition of the norms can be extended to tuples by
	\[
	\norm{(x_1,\dots,x_d)}_p = \begin{cases} \left( \sum_{j=1}^d \tau(|x_j|^p) \right)^{1/p}, & p \in [1,\infty), \\ \max_{j=1,\dots,d} \norm{x_j}, & p = \infty. \end{cases}
	\]
	
	If $(M,\tau)$ is viewed as a non-commutative probability space, then its elements maybe viewed as non-commutative random variables.  In fact, a $d$-tuple $x = (x_1,\dots,x_d) \in M_{\sa}^d$ is the non-commutative analog of an $\bR^d$-valued random variable.  Although one cannot define the law (or probability distribution) of $x$ as a classical measure, we may define its non-commutative law as a certain linear function on a non-commutative polynomial algebra, just like the probability distribution $d$-tuple $X$ of bounded classical random variables defines a map $\bC[t_1,\dots,t_d] \to \bC$ sending $p$ to $\mathbb{E}[p(X)]$
	
	For $d \in \bN$, we let $\bC\ang{t_1,\cdots,t_d}$ be the algebra of noncommutative polynomials in $d$ formal variables $t_1$, \dots, $t_d$, i.e. the free $\bC$-algebra with $d$-generators. We give $\bC\ang{t_{1},\cdots,t_{d}}$ the unique $*$-algebra structure which makes the $t_{j}$ self-adjoint. By universality, if $A$ is any $*$-algebra, and $x=(x_{1},\cdots,x_d)\in A^d$ is a self-adjoint tuple, then there is a unique $*$-homomorphism $\bC\ang{t_{1},\cdots,t_d}\to A$ which sends $t_{j}$ to $x_{j}$. For $p \in \bC\ang{t_{1},\cdots,t_d}$ we use $p(x)$ for the image of $p$ under this $*$-homomorphism.  Given a tracial von Neumann algebra $(M,\tau)$ and $x \in M^d_{\sa}$, we define the \emph{law of $x$}, denoted $\ell_{x}$, to be the linear functional $\ell_{x}\colon \bC\ang{t_1,\cdots,t_d}\to \bC$ given by
	\[
	\ell_{x}(f)=\tau(f(x)).
	\]
	The non-commutative laws can be characterized as follows.
	
	\begin{prop}[{See \cite[Proposition 5.2.14]{AGZ2009}}] \label{prop:NClaws}
		Let $\ell: \bC\ang{t_1,\dots,t_d} \to \bC$ and let $R > 0$.  The following are equivalent:
		\begin{enumerate}[(i)]
			\item There exists a tracial von Neumann algebra $(M,\tau)$ and $x \in M_{\sa}^d$ such that $\ell = \ell_x$ and $\norm{x}_\infty \leq R$.
			\item $\ell$ satisfies the following conditions: \label{item:abstract characterization of law}
			\begin{itemize}
				\item $\ell(1) = 1$,
				\item $\ell(f^*f) \geq 0$ for $f \in \bC\ang{t_1,\dots,t_d}$,
				\item $\ell(fg) = \ell(gf)$ for $f, g \in \bC\ang{t_1,\dots,t_d}$,
				\item $|\ell(t_{i_1} \dots t_{i_k})| \leq R^k$ for all $k \in \bN$ and $i_1$, \dots, $i_k \in \{1,\dots,d\}$.
			\end{itemize}
		\end{enumerate}
	\end{prop}
	
	\begin{rem} \label{rem:lawGNS}
		In fact, the proof of (ii) $\implies$ (i) gives an explicit description of the von Neumann algebra through the Gelfand-Naimark-Segal construction (see \cite[Proposition 5.2.14]{AGZ2009}).  Let $H = L^2(\ell)$ be separation-completion of $\bC\ang{t_1,\dots,t_d}$ with respect to the semi-inner product $\langle{f,g\rangle}_\ell = \ell(f^*g)$.  One can show that multiplication by $t_j$ gives a well-defined, bounded, self-adjoint operator $x_{j}$ on $\mathcal{H}$.  We take $M$ to be the von Neumann algebra generated by $x_1$, \dots, $x_d$, and let $\tau$ be the state corresponding to the vector $1$ in $\mathcal{H}$.  In fact, we will denote $M$ by $\mathrm{W}^*(\ell)$ and we denote by $\pi_\ell$ the unital $*$-homomorphism $\bC\ang{t_1,\dots,t_d}$ sending $t_j$ to $x_j$.
	\end{rem}
	
	Let $\Sigma_{d,R}$ be the set of all linear maps $\bC\ang{t_1,\dots,t_d} \to \bC$ satisfying the equivalent conditions of Proposition \ref{prop:NClaws}.  We equip $\Sigma_{d,R}$ with the weak$^*$ topology, that is, the topology of pointwise convergence on $\bC\ang{t_1,\dots,t_d}$.  It is easy to see that $\Sigma_{d,R}$ is compact and metrizable using Proposition \ref{prop:NClaws} (\ref{item:abstract characterization of law}).  
	
	\subsection{Microstate spaces and $1$-bounded entropy}
	
	Let $(M,\tau)$ be a diffuse tracial von Neumann algebra, and $x\in M_{\sa}^d$ for some $d \in \bN$ with $W^{*}(x)=M$.  Suppose that $\norm{x}_\infty \leq R$.  Following \cite{VoiculescuFreeEntropy2}, for each open set $\mathcal{O}$ of $\Sigma_{d,R}$ and $N \in \bN$, we define
	\[
	\Gamma_R^{(n)}(\cO) = \{X \in \mathbb{M}_n(\bC)_{\sa}^d: \ell_X \in \mathcal{O} \}.
	\]
	When $\mathcal{O}$ is a neighborhood of $\ell_x$, we call $\Gamma_R^{(n)}(\mathcal{O})$ a \emph{microstate space} for $x$.

	Given $d,n\in \bN$, $p\in [1,\infty]$, $\varepsilon>0$ and $\Omega,\Xi\subseteq \bM_{n}(\bC)^{d}$ then $\Xi$ is said to \emph{$(\varepsilon,\norm{\cdot}_p)$-cover $\Omega$} if for every $A\in \Omega$, there is a $B\in \Xi$ with $\|A-B\|_{p}<\varepsilon$. We define the \emph{covering number} of $\Omega\subseteq \bM_{n}(\bC)^{d}$, denote $K_{\varepsilon}(\Omega,\|\cdot\|_{p})$, to be the minimal cardinality of a set that $(\varepsilon,\|\cdot\|_{p})$-covers $\Omega$. We will use covering numbers for different values of $p$ in Section \ref{sec: Property (T)}.
	
	While these covering numbers are natural for many purposes, for unitarily invariant subsets of matrices, it is natural to take the orbital numbers modulo unitary conjugation.
	Given $n \in \bN$, $\varepsilon > 0$ and $\Omega,\Xi \subseteq \mathbb{M}_{n}(\bC)^d$ we say that $\Xi$ \emph{orbitally $(\varepsilon,\norm{\cdot}_2)$-covers} $\Omega$ if for every $A\in\Omega$, there is a $B \in \Xi$ and an $n \times n$ unitary matrix $V$ so that
	\[
	\norm{A-VBV^*}_2 < \varepsilon.
	\]
	We define the \emph{orbital covering number} $K_{\varepsilon}^{\textnormal{orb}}(\Omega,\norm{\cdot}_2)$ as the minimal cardinality of a set of $\Omega_0$ that orbitally $(\varepsilon,\norm{\cdot}_2)$-covers $\Omega$. Since we will usually be concerned with $\|\cdot\|_{2}$-norms we will frequently drop $\|\cdot\|_{2}$ from the notation and use $K_{\varepsilon}^{\textnormal{orb}}(\Omega)$ instead of $K_{\varepsilon}^{\textnormal{orb}}(\Omega,\|\cdot\|_{2})$.
	Let $R\in [0,\infty)$ be such that $\norm{x}_\infty < R$.
	
	The $1$-bounded entropy is defined in terms of the exponential growth rate of covering numbers of $\Gamma_R^{(n)}(\mathcal{O})$ (up to unitary conjugation) as $n \to \infty$ for neighborhoods $\mathcal{O}$ of $\ell_x$.
	For a weak$^{*}$-neighborhood $\mathcal{O}$ Of $\ell_{x}$, we define
	\begin{align*}
		h_{R,\varepsilon}(\mathcal{O})& := \limsup_{n \to\infty}\frac{1}{n^2}\log K_{\varepsilon}^{\orb}(\Gamma_{R}^{(n)}(\mathcal{O})), \\
		h_{R,\varepsilon}(x) &:= \inf_{\mathcal{O} \ni \ell_x} h_{R,\varepsilon}(\mathcal{O}),
	\end{align*}
	where the infimum is over all weak$^{*}$-neighborhoods $\mathcal{O}$ of $\ell_{x}$.  We then define
	\[
	h_R(x) := \sup_{\varepsilon > 0} h_{R,\varepsilon}(x).
	\]
	
	In \cite{Hayes2018}, it is shown that $h_R(x)$ is independent of $R$ provided that $\norm{x}_\infty \leq R$, and hence we may unambiguously denote it by $h(x)$.  In fact, $h(x)$ only depends on the von Neumann algebra generated by $x$.  Hence, for every finitely generated tracial von Neumann algebra $M$, we may define $h(M)$ as $h(x)$ for some generating tuple $x$.  If $M$ is not a factor, then the $1$-bounded entropy depends upon the trace we choose on $M$. We will use $h(M,\tau)$ for a tracial von Neumann algebra $(M,\tau)$ if we wish to emphasize the dependence of the $1$-bounded entropy on $\tau$. However, we will typically suppress this from the notation and just use $h(M)$ unless there is a possibility of confusion.  The definition of $h$ can be extended even to the case where $M$ is not finitely generated; see \cite{Hayes2018}.  However, our results only require a direct appeal to the definition in the finitely generated setting.
	The importance of  $1$-bounded entropy to the study of strong $1$-boundedness is encapsulated by the following result.
	
	\begin{thm}[{See  \cite[Proposition A.16]{Hayes2018}}]
		A tracial von Neumann algebra $M$ is strongly $1$-bounded in the sense of Jung \cite{Jung2007} if and only if $h(M) < \infty$.
	\end{thm}\label{thm:S1B in terms of 1-bounded entropy}
	Because of this result, we will not use Jung's original definition of strongly $1$-bounded \cite{Jung2007} and will instead show that algebras are strongly $1$-bounded by showing that they have finite $1$-bounded entropy.
	
	\subsection{Properties and applications of $1$-bounded entropy}\label{sec:S1B background}
	
	Since we will use general properties of $1$-bounded entropy frequently, we list some of the main ones here. To state these we also need the notion of the \emph{$1$-bounded entropy of $N$ in the presence of $M$}, denoted $h(N:M)$, defined for an inclusion $N\leq M$ of tracial von Neumann algebras. The $1$-bounded entropy in the presence is defined by modifying the definition of $1$-bounded entropy above to only measure the size of the space of microstates for $N$ which have an extension to microstates for $M$. We will not need the precise definition, so we refer the reader to \cite[Definition A.2]{Hayes2018} for the definition. We also use the microstates free entropy dimension $\delta_{0}(x)$ due to Voiculescu \cite[Definition 6.1]{Voiculescu1996} for a tuple $x\in M^{k}_{\sa}$ in a tracial von Neumann algebra $(M,\tau)$.
	Here and throughout the paper, $\mathcal{R}$ refers to the hyperfinite $\textrm{II}_{1}$-factor.   For an inclusion $N\leq M$, we let $\mathcal{N}_{M}^{wq}(N)$ be the set of $u\in \mathcal{U}(M)$ so that $uNu^{*}\cap N$ is diffuse. We assume all inclusions list below are trace preserving.
	We now list some general properties the $1$-bounded entropy here.
	
	\begin{enumerate}[1.]
		\item $h(M)=h(M:M)$ for every diffuse tracial von Neumann algebra $(M,\tau)$,
		\item $h(N:M)\geq 0$ if $N\leq M$ if diffuse and $M$ embeds into an ultrapower of $\mathcal{R}$, and $h(N:M)=-\infty$ if $M$ does not embed into an ultrapower of $\mathcal{R}$. (Exercise from the definitions).
		\item $h(N_{1}:M_{1})\leq h(N_{2}:M_{2})$ if $N_{1}\leq N_{2}\leq M_{2}\leq M_{1}$, if $N_{1}$ is diffuse. \label{I:monotonicity of 1 bounded entropy}(Exercise from the definitions).
		\item $h(N:M)\leq 0$ if $N\leq M$ and $N$ is diffuse and hyperfinite. \label{I:hyperfinite has 1-bdd ent zero} (Exercise from the definitions).
		\item $h(M)=\infty$ if $M = \mathrm{W}^*(x_{1},\cdots,x_{n})$ where $x_{j}\in M_{sa}$ for all $1\leq j\leq n$ and $\delta_{0}(x_{1},\cdots,x_{n})>1$. For example, this applies if $M=L(\bF_{n})$, for $n>1$. (This follows from Theorem \ref{thm:S1B in terms of 1-bounded entropy} and \cite[Corollary 3.5]{Jung2007}).)
		\item $h(N_{1}\vee N_{2}:M)\leq h(N_{1}:M)+h(N_{2}:M)$ if $N_{1},N_{2}\leq M$ and $N_{1}\cap N_{2}$ is diffuse. (See \cite[Lemma A.12]{Hayes2018}.)\label{I:subadditivity of 1 bdd ent}
		\item Suppose that $(N_{\alpha})_{\alpha}$ is an increasing chain of diffuse von Neumann subalgebras of a von Neumann algebra $M$. Then
		\[h\left(\bigvee_{\alpha}N_{\alpha}:M\right)=\sup_{\alpha}h(N_{\alpha}:M).\]
		\label{I:increasing limits of 1bdd ent first variable} (See \cite[Lemma A.10]{Hayes2018}.)
		\item $h(N:M)=h(N:M^{\omega})$ if $N\leq M$ is diffuse, and $\omega$ is a free ultrafilter on an infinite set. (See \cite[Proposition 4.5]{Hayes2018}.) \label{I:omegafying in the second variable}
		\item $h(\mathrm{W}^*(\mathcal{N}_{M}^{wq}(N)):M)=h(N:M)$ if $N\leq M$ is diffuse.  Here $\mathcal{N}_{M}^{wq}(N)=\{u\in \mathcal{U}(M):uN u^{*}=N\mbox{ is diffuse}\}$. (see \cite[Theorem 3.8 and Proposition 3.2]{Hayes2018}.) \label{item:preservation under normalizers}
		\item Let $I$ be a countable set, and $M=\bigoplus_{i\in I}M_{i}$ with $M_{i}$ diffuse for all $i$. Suppose that $\tau$ is a faithful trace on $M$, and that $\lambda_{i}$ is the trace of the identity on $M_{i}$. Endow $M_{i}$ with the trace $\tau_{i}=\frac{\tau|_{M_{i}}}{\lambda_{i}}$. Then
		\[h(M,\tau)\leq \sum_{i}\lambda_{i}^{2}h(M_{i},\tau_{i}).\]
		(See \cite[Proposition A.13]{Hayes2018}.)
		\label{item:direct sum inequality}
	\end{enumerate}
	The above axioms imply that all the von Neumann algebras in the following list have nonpositive $1$-bounded entropy:
	\begin{itemize}
		\item hyperfinite algebras,
		\item factors with Property Gamma,
		\item non-prime von Neumann algebras,
		\item algebras with diffuse center,
		\item algebras with diffuse, regular hyperfinite subalgebras (e.g. if the algebra has a Cartan subalgebra),
	\end{itemize}
	For proofs, see \cite[Section 1.2]{FreePinsker}. The class of algebras with nonpositive $1$-bounded entropy is also closed under direct sums (by item \ref{item:direct sum inequality}). Additionally, by item \ref{item:preservation under normalizers} if $N\leq M$ is regular (i.e. $W^{*}(\mathcal{N}_{M}(N))=M$), and $h(N:M)\leq 0$, then $h(M)\leq 0$. This applies if $h(N)\leq 0$. In particular, if $N$ is hyperfinite, or has Gamma, or is not prime, or has diffuse center, and if $N$ is regular in $M$, then $h(M)\leq 0$.

	\subsection{Property (T)}
	
	Property (T) for groups, due to Kazhdan \cite{KazhdanTDef}, is defined as follows.  If $\pi: G \to \mathcal{U}(\mathcal{H})$ is a unitary representation of the group $G$ on a Hilbert space, then a vector $\xi \in \mathcal{H}$ is said to be \emph{invariant} if $\pi(g) \xi = \xi$ for all $g \in G$.  We say that the group $G$ has \emph{Property (T)} if there exists a finite $F \subseteq G$ and $\delta > 0$ such that for every unitary representation $\pi$ on $\mathcal{H}$, if there exists a nonzero $\xi \in \mathcal{H}$ such that $\norm{\pi(g) \xi - \xi} < \delta$ for all $g \in F$, then $\mathcal{H}$ contains a nonzero invariant vector.
	
	It turns out that if $G$ has Property (T), then there exists a finite set $F \subseteq G$ and some $\gamma > 0$, such that for every unitary representation $\pi$ of $G$ on a Hilbert space $\mathcal{H}$, we have
	\[
	\norm{\xi - P_{\operatorname{invariant}}\xi} \leq \frac{1}{\gamma} \sum_{g \in F} \norm{\pi(g)\xi - \xi},
	\]
	where $P_{\operatorname{invariant}}: H \to H$ is the projection onto the subspace of invariant vectors.  Such a set $F$ is called a \emph{Kazhdan set} and $\gamma$ is called the \emph{Kazhdan constant}.  For a proof, see \cite[Definition 1.1.3, Proposition 1.1.9]{BHV}.  It is sometimes notationally convenient to list the elements of $F$ as a tuple $(g_1,\dots,g_{|F|})$, which we call a \emph{Kazhdan tuple}.
	
	The von Neumann algebraic analogue of Property (T), introduced by Connes and Jones \cite{ConnesJones} for $\mathrm{II}_1$ factors and Popa \cite{PopaL2Betti} for general tracial von Neumann algebras, uses Hilbert bimodules over a tracial von Neumann algebra $M$ in place of group representations.
	
	\begin{defn}~
		\begin{enumerate}[(i)]
			\item If $M$ is a von Neumann algebra, then a \emph{Hilbert $M$-$M$-bimodule} is a Hilbert space equipped with normal left and right actions of $M$.
			\item A vector $\xi$ in a Hilbert $M$-$M$ bimodule $\mathcal{H}$ is \emph{central} if $x \xi = \xi x$ for all $x \in M$.
			\item If $(M,\tau)$ is a tracial von Neumann algebra, then a vector $\xi$ in an $M$-$M$-bimodule $M$ is \emph{bitracial} if $\ang{\xi, x\xi} = \tau(x) = \ang{\xi, \xi x}$ for all $x \in M$.
		\end{enumerate}
	\end{defn}
	
	There is analogy between central vectors in a Hilbert $M$-$M$-bimodule and invariant vectors in a group representation motivated by the way in which representations of a discrete group $G$ naturally give rise to bimodules over the group von Neumann algebra $L(G)$.  Let $\lambda$ and $\rho$ denote the left and right regular representations of $G$ on $\ell^2(G)$.  If $\pi$ is a representation of $G$ on $\mathcal{H}$, then $\pi \otimes \lambda$ is a left representation on $H \otimes \ell^2(G)$ and $1 \otimes \rho$ is a right representation on $H \otimes \ell^2(G)$.  These representations extend to normal left/right actions of $L(G)$, making $H \otimes \ell^2(G)$ into a Hilbert $L(G)$-$L(G)$-bimodule.  If $\xi$ is an invariant unit vector in $\mathcal{H}$, then $\xi \otimes \delta_e$ is a central and bitracial vector in $H \otimes \ell^2(G)$.  Hence, Property (T) for tracial von Neumann algebras is defined as follows.
	
	\begin{defn}[{\cite{ConnesJones,PopaL2Betti}}]
		A tracial von Neumann algebra $(M,\tau)$ has \emph{Property (T)} if for every $\varepsilon > 0$, there is a finite $F \subseteq M$ and $\delta > 0$, such that for every bitracial vector $\xi$ in a Hilbert $M$-$M$-bimodule $\mathcal{H}$, if $\sum_{x \in F} \norm{x\xi - \xi x} < \delta$, then there exists a central $\eta \in \mathcal{H}$ with $\norm{\xi - \eta} < \varepsilon$.
	\end{defn}
	
	It turns out that a group $G$ has Property (T) if and only if $L(G)$ has Property (T) \cite[Theorem 2]{ConnesJones}. One can also formulate a von Neumann algebraic version of Kazhdan tuples as follows.
	
	\begin{defn}
		Let $(M,\tau)$ be a tracial von Neumann algebra.  We say that $x = (x_1,\dots,x_d) \in M_{\sa}^d$ is a \emph{(self-adjoint) Kazhdan tuple} if there exists $\gamma > 0$ such that for every Hilbert $M$-$M$-bimodule $\mathcal{H}$ and $\xi \in \mathcal{H}$, we have
		\[
		\norm{\xi - P_{\textnormal{central}} \xi} \leq \frac{1}{\gamma} \norm{x\xi - \xi x}_{H^{\oplus d}},
		\]
		where $x \xi = (x_1 \xi, \dots, x_d \xi)$ and $\xi x = (\xi x_1,\dots, \xi x_d)$, and where $P_{\textnormal{central}}$ is the projection of $H$ onto the subspace of central vectors.  In this case, we call $\gamma$ the \emph{Kazhdan constant} associated to the Kazhdan tuple $x$.
	\end{defn}
	
	Unlike the group case, every tracial von Neumann algebra with Property (T) need not admit a Kazhdan tuple (the finite set $F$ in the definition of Property (T) may depend \emph{a priori} on $\varepsilon$).  However, it is known for several natural classes of examples.
	
	\begin{lemm}~
		\begin{enumerate}[(i)]
			\item If $M$ is a Property (T) von Neumann algebra with finite-dimensional center, then $M$ has a Kazhdan tuple. \label{item:fd center Kazhdan tuple}
			\item  If $G$ is a Property (T) group, and $\pi\colon G\to \mathcal{U}(\mathcal{H})$ is a projective unitary representation such that $W^{*}(\pi(G))$ is finite, then $W^{*}(\pi(G))$ has a finite Kazhdan tuple. \label{item:group Kazhdan tuple}
		\end{enumerate}
	\end{lemm}
	
	\begin{proof}
		(\ref{item:fd center Kazhdan tuple}): Cutting $M$ down by the minimal central projections in $M$, we see that $M$ is a direct sum of $\textrm{II}_{1}$-factors. Each such factor has Property (T) by \cite[Proposition 4.7.2]{PopaL2Betti}, and thus admits a Kazhdan tuple by \cite[Proposition 1]{ConnesJones}. Combining these Kazhdan tuples gives a Kazhdan tuple for $M$.
		
		(\ref{item:group Kazhdan tuple}): Choose a Kazhdan set $\{g_{1},\cdots,g_{k}\}\subseteq G$ for $G$. Let
		\[x=(\Re(\pi(g_{1})),\Re(\pi(g_{2})),\dots,\Re(\pi(g_{k})),\Im(\pi(g_{1})),\dots,\Im(\pi(g_{k})))\in M_{\sa}^{2k}.\]
		Even though $\pi$ may not be an honest representation, note that if $\mathcal{H}$ is an $M$-$M$ bimodule, then we do have a representation of $G$ on $\mathcal{H}$ given by conjugating by $\pi(G)$. Applying \cite[Proposition 1.1.9]{BHV} to this representation, we see that $x$ is a Kazhdan tuple for $W^{*}(\pi(G))$.
	\end{proof}
	
	
	It is a standard exercise to show that a Kazhdan tuple for a group $G$ must generate the group.  A well-known result of Popa \cite[Theorem 4.4.1]{PopaCorr} shows that the same is true for Kazhdan tuples associated to tracial von Neumann algebras.  For the reader's convenience, we recall the proof here.
	
	\begin{lemm} \label{lemm: Kazhdan tuple generates}
		If $(M,\tau)$ is a tracial von Neumann algebra and $x = (x_1,\dots,x_d)$ is a Kazhdan tuple, then $x$ generates $M$ as a von Neumann algebra.
	\end{lemm}
	
	\begin{proof}
		Consider the standard (GNS) representation of $M$ on $L^2(M,\tau)$.  Let $N = \mathrm{W}^*(x) \subseteq M$.  Let $e \in B(L^2(M,\tau))$ be the orthogonal projection onto $L^2(N,\tau|_N) \subseteq L^2(M,\tau)$.  The von Neumann algebra $M_1 = \ang{M,e} \subseteq B(L^2(M,\tau))$ is called the basic construction for $N \subseteq M$.  It is well known that $M \cap \{e\}' = N$ and that $M$ has a semi-finite trace $\Tr$ such that $\Tr(aeb) = \tau(E_N(a) E_N(b))$ for $a, b \in M$ by \cite{Jones83}.  Since $M$ embeds into $M_1$, we may regard $L^2(M_1,\Tr)$ as an $M$-$M$ bimodule.  The element $e \in L^2(M_1,\Tr)$ satisfies $ae = ea$ for $a \in N$ and in particular this holds for $a = x_1$, \dots, $x_d$.  Since $X$ is a Kazhdan tuple for $M$, we have $ae = ea$ for all $a \in M$, which implies that $M = N$.
	\end{proof}
	
	\subsection{Ultraproducts of matrix algebras} \label{subsec:ultraproduct}
	
	At several points, we use ultraproducts of matrix algebras, so we recall the relevant background as well as the connection between ultraproducts and microstate spaces.
	
	
	\begin{defn}
		Let $J$ be an infinite set. A \emph{free ultrafilter} is a unital homomorphism
		\[\omega\colon \ell^{\infty}(J)/c_{0}(J)\to \bC.\]
		For every $(a_{n})_{n\in J}\in \ell^{\infty}(J)$ we will use $\lim_{n\to\omega}a_{n}$ for $\omega((a_{n})_{n\in J}+c_{0}(J))$. 
	\end{defn}
	
	It follows from \cite[Proposition VIII.1.12]{Conway} that $\lim_{n\to\omega}$ is a $*$-homomorphism. It thus preserves inequalities, and commutes with complex conjugation, by definition $\lim_{n\to\omega}$ preserves sums and products. By Gelfand-Naimark duality \cite[Theorem VIII.2.1]{Conway}, free ultrafilters exist in abundance (in fact, there are enough to separate points in $\ell^{\infty}(J)/c_{0}(J)$). We use $\beta J\setminus J$ for the space of free ultrafilters on $J$. 
	
	Recall that if $\omega$ is a free ultrafilter on $\bN$, then the tracial ultraproduct of $\bM_{n}(\bC)$ with respect to $\omega$ is given by
	\[
	\prod_{n\to\omega}\bM_{n}(\bC)=\frac{\{(x_{k})_{n}\in \prod_{n}\bM_{n}(\bC):\sup_{n}\|x_{n}\|_{\infty}<\infty\}}{\{(x_{n})_{n}\in \prod_{n}\bM_{n}(\bC):\sup_{n}\|x_{n}\|_{\infty}<\infty,\lim_{n\to\omega}\|x_{n}\|_{2}=0\}}.
	\]
	If $(x_{n})_{n}\in \prod_{n}\bM_{n}(\bC)$ and $\sup_{n}\|x_{n}\|_{\infty}<\infty$, we let $[x_{n}]_{n}$ be the image of $(x_{n})_{n}$ under the natural quotient map
	\[
	\{(x_{n})_{n}\in \prod_{n}\bM_{n}(\bC):\sup_{n}\|x_{n}\|_{\infty}<\infty\}\to \prod_{n\to\omega}\bM_{n}(\bC).
	\]
	It can be shown (see \cite[Lemma A.9]{BrownOzawa2008}) that this is a tracial von Neumann algebra, with trace given by
	\[\tau_{\omega}((x_{n})_{n\to\omega})=\lim_{k\to\omega}\tr_n(x_{n}).\]
	We say that a tracial von Neumann algebra $(M,\tau)$ is \emph{Connes-embeddable} if every von Neumann subalgebra with separable predual admits a trace-preserving embedding into a tracial ultraproduct of matrices.
	
	Suppose that $d \in \bN$, and $R>0$. If $(\ell_{n})_{n}$ is a sequence in $\Sigma_{d,R}$, and $\ell\in \Sigma_{d,R}$ we say that $\lim_{n\to\omega}\ell_{n}=\ell$ if $\lim_{n\to\omega}\ell_{n}(P)=\ell(P)$ for all $P\in \bC\ang{t_{1},\cdots,t_{d}}$.  Suppose $(M,\tau)$ is a tracial von Neumann algebra and $x\in M_{\sa}^{d}$. If $X^{(n)}\in \bM_{n}(\bC)_{\sa}^{d}$ satisfies $\lim_{n\to\omega}\ell_{X^{(n)}}=\ell_{x}$, then there is a unique trace-preserving embedding $\pi\colon M\to \prod_{n\to\omega}\bM_{n}(\bC)$ which satisfies
	\[
	\pi(x_{j})=[X_{j}^{(n)}]_{n}\textnormal{ for $j=1,\cdots,d$;}
	\]
	(see for instance \cite[Lemma 5.10]{GJNS2021} for further detail).  Because of this, we may heuristically think of the microstate spaces $\Gamma^{(n)}_{R}(\mathcal{O})$ as $\mathcal{O}\to \ell_{x},n\to\infty$ as parameterizing the space of embeddings of $M$ into an ultraproduct of matrices.
	
	See \cite{JungTubularity} and \cite[Sections 1.2, 1.3]{ScottSri2019ultraproduct} for further discussion of the connections between ultraproducts and microstate spaces.  These concepts also relate to random matrix theory; see for instance \cite{VoiculescuFreeEntropy2,Voiculescu1998,DJS2005,BDJ2008,FreePinsker}.

	\section{Proof of Theorem \ref{thm: Property T}} \label{sec: Property (T)}
	
	
	As in Theorem \ref{thm: Property T}, we consider a tracial von Neumann algebra $(M,\tau)$ together with a Kazhdan tuple $x = (x_1,\dots,x_d)$.  The following lemma is a refinement of \cite[Theorem 1.1]{JungS}.  It shows roughly that microstates for our Kazhdan tuple are unitarily conjugate except for a piece that lies under a projection of small trace.
	The core idea for the proof of this lemma is the corresponding statement about embeddings into ultraproducts of matrices: if two embeddings of a Property (T) algebra are close on $x$, then they are unitarily conjugate except on a small corner of $M$. Further, the size of that corner can be controlled in terms of how close they are on $x$.

	\begin{lemm} \label{lemm: property T microstates}
		Let $(M,\tau)$ be a tracial von Neumann algebra and let $x = (x_1,\dots,x_d) \in M_{\sa}^d$ be a Kazhdan tuple with Kazhdan constant $\gamma$ and suppose $\norm{x}_\infty < R$.  Then for every $\varepsilon > 0$ and $\delta > 0$, there exists a neighborhood $\cO$ of $\ell_x$ and $n_0 \in \bN$ such that for all $n \geq n_0$, if $Y, Z \in \Gamma_R^{(n)}(\cO)$ with $\norm{Y - Z}_2 \leq \varepsilon$, then there exists a unitary $U$ in $\mathbb{M}_n(\bC)$ and a projection $P$ such that
		\[
		\norm{(U^*YU - Z)(1 - P)}_2 < \delta, \qquad \tr_N(P) < \frac{\varepsilon}{\gamma} \left(2 + \frac{\varepsilon}{\gamma} \right) + \delta.
		\]
	\end{lemm}
	
	\begin{proof}
		Write
		\[
		\Delta(Y,Z,U,P) = \max\left( \norm{(U^*YU - Z)(1 - P)}_2 , \tr_N(P) - (\varepsilon/\gamma)(2 + \varepsilon/\gamma) \right).
		\]
		Fix $\varepsilon > 0$ and $\delta > 0$.  Suppose for contradiction that the claim fails.  Fix a sequence of neighborhoods $\cO_k$ shrinking to $\mu$.  By our assumption, for each $k$ there must exist arbitrarily large $N \in \bN$ such that the claim about $Y$, $Z \in \Gamma_R^{(n)}(\cO_k)$ fails.  Hence, we can choose $n_k \in \bN$ with $n_{k+1} > n_k$ such that there exist $Y^{(n_k)}$, $Z^{(n_k)} \in \Gamma_R^{(n_k)}(\cO_k)$ with $\norm{Y^{(n_k)} - Z^{(n_k)}}_2 \leq \varepsilon$ such that
		\[
		\liminf_{k \to \infty} \inf \{\Delta(Y^{(n_k)},Z^{(n_k)},U,P): U \text{ unitary, } P \text{ projection in } \mathbb{M}_n(\bC) \} \geq \delta.
		\]
		Now choose a free ultrafilter $\omega\in \beta \bN\setminus \bN$ so that $\lim_{n\to\omega}1_{\{n_{k}:k\in \bN\}}(n)=1.$
		For notational convenience, extend $Y^{(n_k)}$ and $Z^{(n_k)}$ to sequences $Y^{(n)}$ and $Z^{(n)}$ defined for all $n \in \bN$ (however, the values for $n \not \in \{n_k: k \in \bN\}$ will not matter).  Then
		\[
		\lim_{n \to \omega} \inf \{\Delta(Y^{(n)},Z^{(n)},U,P): U \text{ unitary, } P \text{ projection in } \mathbb{M}_n(\bC) \} \geq \delta.
		\]
		and
		\[
		\lim_{n \to \omega} \ell_{Y^{(n)}} = \lim_{n \to \omega} \ell_{Z^{(n)}} = \ell_{x}.
		\]
		Let
		\[
		(\mathcal{M},\tr) = \prod_{n \to \omega} (\mathbb{M}_n(\bC),\tr_n),
		\]
		that is, the tracial $\mathrm{W}^*$ ultraproduct of matrix algebras with respect to the ultrafilter $\omega$.
		Let $y = [Y^{(n)}]_{n\in\bN}$ and $z = [Z^{(n)}]_{n \in \bN}$ in $\mathcal{M}_{\sa}^d$.  Then $y$ and $z$ have the same law as $x$ and therefore, there are trace-preserving embeddings $\pi_1, \pi_2: (M,\tau) \to (\cM,\tr)$ such that $\pi_1(x) = y$ and $\pi_2(x) = z$.
		We can make $L^{2}(\cM,\tr)$ into an $M$-$M$ bimodule by letting $M$ act on the left by $\pi_{1}$ and on the right by $\pi_{2}$.
		Let $T: L^2(\cM,\tr) \to L^2(\cM,\tr)^d$ be the operator $T(a) = ya - az$.
		By definition of the Kazhdan constant $\gamma$, we have
		\[
		\norm{1 - P_{\ker(T)}(1)}_2 \leq \frac{1}{\gamma} \norm{T(1)}_2 = \frac{1}{\gamma} \norm{y - z}_2 \leq \frac{\varepsilon}{\gamma}.
		\]
		
		Let $a = P_{\ker(T)}(1) \in L^2(\mathcal{M},\tr)$, so that $ya = az$ by definition of $T$.  Here we consider $a$ as an affiliated operator to $\mathcal{M}$; (see \cite[Appendix F]{BrownOzawa2008}) and as such it has a polar decomposition $a = v|a|$ where $|a| = (a^*a)^{1/2} \in L^2(\mathcal{M},\tr)$ and $v$ is a partial isometry with range equal to the closure of the range of $a$ (see \cite[Theorem VII.32]{ReedUndSimon}).
		It is not hard to show that $yv = vz$ (see \cite[Lemma 5]{SorinInterwiningBBM}), and it follows that $v^*v$ commutes with $z$.
		Furthermore,
		\begin{align*}
			1 - \tr(v^*v) &= \tr(P_{\ker(a)}) \\
			&\leq \norm{a^*a - 1}_1 \\
			&\leq \norm{(a^* - 1)a}_1 + \norm{a-1}_1 \\
			&\leq \norm{a^*-1}_2 (1 + \norm{a - 1}_2) + \norm{a - 1}_2 \\
			&\leq \frac{\varepsilon}{\gamma} \left(2 +  \frac{\varepsilon}{\gamma} \right).
		\end{align*}
		Let $p = 1 - v^*v$.  Since $\mathcal{M}$ is a finite von Neumann algebra, there exists a partial isometry sending $1 - v^*v$ to $1 - vv^*$, and therefore, there is a unitary $u$ such that $v = u(1 - p)$.  Since $1 - p = v^*v$ commutes with $z$, we have
		\[
		(u^*yu - z)(1 - p) = u^*yu(1 - p) - u^*u(1 - p)z = u^*(yv - vz) = 0.
		\]
		It is a standard exercise to show that there exists a sequence of projections $(P^{(n)})_{n \in \bN}$ and a sequence of unitaries $(U^{(n)})_{n \in \bN}$ such that $p = [P^{(n)}]_{n \in \bN}$ and $u = [U^{(n)}]_{n \in \bN}$.  We also have
		\[
		\lim_{n \to \omega} \norm{((U^{(n)})^*Y^{(n)}U^{(n)} - Z^{(n)})(1 - P^{(n)}}_2 = \norm{(u^*yu - z)(1-p)}_2 = 0
		\]
		and
		\[
		\lim_{n \to \omega} \tr_n(P^{(n)}) = \tr(p) \leq \frac{\varepsilon}{\gamma} \left( 2 + \frac{\varepsilon}{\gamma} \right).
		\]
		Therefore,
		\[
		\lim_{n \to \omega} \Delta(Y^{(n)},Z^{(n)},U^{(n)},P^{(n)}) = 0,
		\]
		which contradicts our choice of $Y^{(n)}$ and $Z^{(n)}$, and thus the argument is complete.
	\end{proof}
	
	Now that we know that microstates are conjugate up to a small projection, in order to control the covering number of the microstate space, we need to estimate the covering numbers of the space of these projections.  We rely on an estimate of Szarek \cite{Szarek} on the covering numbers of Grassmannians, which we state in the following form.
	
	\begin{lemm} \label{lemm: Szarek}
		There exists a universal constant $C$ such that for $t \geq 0$,
		\begin{align*}
			K_{\varepsilon}(\{P \in \mathbb{M}_n(\bC) \text{ projection, } \tr_n(P) \leq t\},\|\cdot\|_{\infty}) &= K_{\varepsilon}(\{P \in \mathbb{M}_n(\bC) \text{ projection, } \tr_n(P) \geq 1 - t \}, \norm{\cdot}_\infty) \\
			&\leq (1 + nt) \left( \frac{C}{\varepsilon} \right)^{2n^2t}.
		\end{align*}
	\end{lemm}
	
	\begin{proof}
		First, the two covering numbers are equal because we can make the substitution $P \mapsto 1 - P$.  Thus, it suffices to estimate the first one.  For $\ell \in \bN\cup\{0\}$, let $G(\ell,n-\ell)$ be the subset of $\mathbb{M}_n(\bC)$ consisting of rank $\ell$ orthogonal projections. By \cite[Theorem 8 and Remark (ii) below it]{Szarek}, there is a uniform $C > 0$ so that
		\[
		K_{\varepsilon}(G(\ell,n-\ell), \norm{\cdot}_\infty)\leq \left(\frac{C}{\varepsilon}\right)^{2\ell(n-\ell)}.
		\]
		Note that
		\[
		\{P \in \mathbb{M}_n(\bC) \text{ projection } \tr_n(P) \leq t\} = \bigcup_{\substack{\ell \in [n] \\ \ell \leq nt}} G(\ell,n-\ell).
		\]
		The number of terms in the union is the ceiling of $n\min(1,t)$ which is bounded by $1 + nt$. The covering number of each of the individual sets can be bounded by $(C / \varepsilon)^{2\ell(n-\ell)} \leq (C / \varepsilon)^{2n^2t}$ since $\ell \leq nt$ and $n - \ell \leq n$.
	\end{proof}
	
	Next, we combine Lemma \ref{lemm: property T microstates} and Lemma \ref{lemm: Szarek} to obtain the following estimate, which bounds the $\eta$-covering numbers for microstate spaces in terms of the $\varepsilon$-covering numbers for $\eta \leq \epsilon$.  We will then conclude the proof of the theorem by iterating this estimate.
	
	\begin{lemm} \label{lemm: property T iterative estimate}
		Let $(M,\tau)$ be a tracial von Neumann algebra and suppose $x \in M_{\sa}^d$ is a Kazhdan tuple with $\norm{x}_\infty < R$ and with Kazhdan constant $\gamma$.  Let $0 < \eta \leq \varepsilon < \gamma/2$.  Then
		\[
		h_{R,\eta}(x) \leq h_{R,\varepsilon}(x) + \frac{12(d+1)\varepsilon}{\gamma} \log \frac{CRd^{1/2}}{\eta}
		\]
		where $C$ is a positive constant.
	\end{lemm}
	
	\begin{proof}
		Fix $\delta > 0$ with $\delta < \eta/3$ and $(2\varepsilon / \gamma)(2 + 2\varepsilon / \gamma) + \delta < 6\varepsilon / \gamma$ (the latter being possible because $\varepsilon < \gamma$).  Let $\cO$ and $n_0$ be as in Lemma \ref{lemm: property T microstates} for the constants $\delta$ and $2\varepsilon$ (rather than $\delta$ and $\varepsilon$), and assume that $n \geq n_0$.  For ease of notation, let $t = 6 \varepsilon / \gamma$.
		
		By Lemma \ref{lemm: Szarek}, choose a set $\Xi$ of projections of rank less than or equal to $nt$ such that every projection $P$ with $\tr_n(P) \leq t$ satisfies $\norm{P - Q}_\infty < \eta / (6Rd^{1/2})$, such that $|\Xi| \leq (1 + nt)(6C_1Rd^{1/2}/\eta)^{2n^2 t}$.  For each $Q \in \Xi$ with rank $\ell$, the space $\mathbb{M}_n(\bC)^d Q$ has real dimension $2n \ell d \leq 2 n^2 t d$ and therefore there exists $E_Q \subseteq \mathbb{M}_n(\bC)^d Q$ that $(\eta/3,\norm{\cdot}_2)$-covers $\overline{B}_{\mathbb{M}_n(\bC)^d Q,\norm{\cdot}_2}(0,2Rd^{1/2})$ and satisfies
		\[
		|E_Q| \leq \left( \frac{6C_2Rd^{1/2}}{\eta} \right)^{2n^2 t d}.
		\]
		Finally, fix $\Omega \subseteq \Gamma_R^{(n)}(\cO)$ that orbitally $(2\varepsilon,\norm{\cdot}_2)$-covers $\Gamma_R^{(n)}(\cO)$ and satisfies $|\Omega| \leq K_{\varepsilon}(\Gamma_R^{(n)}(\cO),\norm{\cdot}_2)$.
		
		We claim that $\Omega' = \bigcup_{Q \in \Xi} (\Omega + E_Q)$ is an $(\eta,\norm{\cdot}_2)$-covering of $\Gamma_R^{(n)}(\cO)$.  To see this, suppose $Y \in \Gamma_R^{(n)}(\cO)$.  Then there exists $Z \in \Omega$ and a unitary $V$ such that $\norm{Y - VZV^*}_2 <  2\varepsilon$, hence $\norm{V^*YV - Z}_2 < 2 \varepsilon$.  By our choice of $n_0$ and $\cO$, there exists a projection $P$ and a unitary $U$ such that
		\[
		\norm{(U^*V^*YVU - Z)(1 - P)}_2 < \delta < \frac{\eta}{3}, \qquad
		\tr_n(P) < (2\varepsilon/\gamma)(2 + 2\varepsilon/\gamma) + \delta \leq t.
		\]
		To simplify notation, let us rename $VU$ to $U$, so that $\norm{(U^*YU - Z)(1 - P)}_2 < \eta/3$.  Fix a projection $Q \in \Xi$ such that $\norm{P - Q}_\infty < \eta / (6Rd^{1/2})$.  Note that
		\[
		\norm{(U^*YU - Z)(P-Q)}_2 \leq \norm{U^*YU - Z}_2 \norm{P - Q}_\infty \leq 2Rd^{1/2} \norm{P - Q}_\infty < \frac{\eta}{3}.
		\]
		Moreover,
		\[
		\norm{(U^*YU - Z)Q}_2 \leq d^{1/2} \norm{U^*YU - Z}_\infty \norm{Q}_\infty \leq 2R d^{1/2},
		\]
		hence there exists $W \in E_Q$ such that
		\[
		\norm{(U^*YU - Z)Q - W}_2 < \frac{\eta}{3}
		\]
		It follows that
		\begin{align*}
			\norm{U^*YU - (Z + W)}_2 &\leq \norm{(U^*YU - Z)(1 - P)}_2 + \norm{(U^*YU - Z)(P - Q)}_2 + \norm{(U^*YU - Z)Q - W}_2 \\
			&< \frac{\eta}{3} + \frac{\eta}{3} + \frac{\eta}{3} = \eta.
		\end{align*}
		This shows that $\Omega'$ orbitally $(\eta,\norm{\cdot}_2)$-covers $\Gamma_R^{(n)}(\cO)$ as desired.
		
		Therefore,
		\begin{align*}
			K_\eta^{\orb}(\Gamma_R^{(n)}(\cO)) &\leq |\Omega| \sum_{Q \in \Xi} |E_Q| \\
			&\leq K_{\varepsilon}^{\orb}(\Gamma_R^{(n)}(\cO)) (1 + nt) \left( \frac{6C_1Rd^{1/2}}{\eta}\right)^{2n^2t} \left( \frac{6C_2Rd^{1/2}}{\eta} \right)^{2n^2 td} \\
			&\leq K_{\varepsilon}^{\orb}(\Gamma_R^{(n)}(\cO)) (1 + nt) \left(\frac{ C_3Rd^{1/2}}{\eta}\right)^{2n^2t(d+1)}
		\end{align*}
		where $C_3 = 6 \max(C_1,C_2)$.  After applying $\limsup_{n \to \infty} (1/n^2) \log$ to both sides, we obtain
		\[
		h_{R,\eta}(\cO) \leq h_{R,\varepsilon}(\cO) + 2t(d+1) \log \frac{C_3 R d^{1/2}}{\eta}
		\]
		Substitute back $t = 6 \varepsilon / \gamma$ and letting $\cO$ shrink to $\ell_x$, we obtain the asserted result.
	\end{proof}
	
	Now we can finish the proof of Theorem \ref{thm: Property T} by iterating Lemma \ref{lemm: property T iterative estimate}.
	
	\begin{proof}[Proof of Theorem \ref{thm: Property T}]
		Let $(M,\tau)$ be a tracial $\mathrm{W}^*$-algebra with a Kazhdan tuple $x = (x_1,\dots,x_d)$ with Kazhdan constant $\gamma$ and $\norm{x}_\infty < R$.  For $\varepsilon < \min(\gamma,1)$, we can take $\eta = \varepsilon^2$ in the previous lemma and obtain
		\[
		h_{R,\varepsilon^2}(x) \leq h_{R,\varepsilon}(x) + \frac{12(d+1) \varepsilon}{\gamma} \log \frac{CRd^{1/2}}{\varepsilon^2}.
		\]
		Iterating this estimate, we have that for $\varepsilon < \min(1,\gamma)$,
		\[
		h_{R,\varepsilon^{2^k}}(x) \leq h_{R,\varepsilon}(x) + \frac{12(d+1)}{\gamma} \sum_{j=0}^{k-1} \varepsilon^{2^j} \log \frac{CRd^{1/2}}{\varepsilon^{2^{j+1}}}.
		\]
		Fix $\varepsilon \in (0,\min(1,\gamma))$.  As $k \to \infty$, we have $h_{R,\varepsilon^{2^k}}(x) \to h(M)$.  Therefore,
		\[
		h(M) \leq h_{R,\varepsilon}(x) + \frac{12(d+1)}{\gamma} \sum_{j=0}^\infty \varepsilon^{2^j}\left(\log (CRd^{1/2}) + 2^{j+1} \log \frac{1}{\varepsilon} \right).
		\]
		Since the sum on the right-hand side converges, we have $h(M) < \infty$, and thus $M$ is strongly $1$-bounded.
	\end{proof}

	\section{Direct sums, amplification, and strong $1$-boundedness}\label{sec:d sums and matrix amp}
	
	The main goals of this section are to prove Proposition \ref{prop: dichotomy} and show that our Theorem \ref{thm: Property T} implies Jung and Shlyakhtenko's result that tracial von Neumann algebras with property (T) have free entropy dimension at most $1$ for every finite generating set \cite{JungS} using previously known facts about free entropy dimension.  To these ends, we study the behavior of $1$-bounded entropy under direct sums and matrix amplification.
	
	\subsection{Direct sums and strong $1$-boundedness} \label{subsec: direct sums}
	
	In order to prove Propositions \ref{prop: dichotomy} and \ref{prop: amplification}, we must understand the behavior of $h$ under direct sums of tracial von Neumann algebras.  Let $J$ be a countable index set and $(\lambda_j)_{j \in J} \in (0,\infty)^J$ such that $\sum_{j \in J} \lambda_j = 1$, and for each $j \in J$, let $(M_j,\tau_j)$ be a tracial von Neumann algebra.  Then
	\[
	\bigoplus_{j \in J} \lambda_j(M_j,\tau_j)
	\]
	is defined as the tracial von Neumann algebra $(M,\tau)$ where
	\[
	M = \{(x_j)_{j \in J}: \sup_j \norm{x_j} < \infty\}
	\]
	and
	\[
	\tau((x_j)_{j \in J}) = \sum_{j \in J} \lambda_j \tau_j(x_j).
	\]
	It is shown in \cite[Proposition A.13(i)]{Hayes2018} that
	\begin{equation} \label{eq: direct sum upper bound}
		h(M,\tau) \leq \sum_{j \in J} \lambda_j^2 h(M_j,\tau_j).
	\end{equation}
	
	Our goal in this section is to give a corresponding lower bound for $h(M,\tau)$ (see Lemma \ref{lemm: h direct sum}).  As with many results in free entropy theory, we run into the issue that we do not know whether using a $\liminf$ instead of $\limsup$ in the definition of $h$ would yield the same result.  Thus, we will need to use the ``$\liminf$ version'' of $1$-bounded entropy $\underline{h}$ in our formula.  It is also convenient for the proof to use the description of $1$-bounded entropy in terms of relative microstates rather than unitary orbits.  We therefore recall the following definitions.
	
	\begin{defn}
		Let $(M,\tau)$ be a tracial von Neumann algebra, $l\in \bN$ and $y\in M_{\sa}^{l}$. A \emph{microstates sequence} for $y$ is a sequence $(Y^{(n)})\in \prod_{n}\bM_{n}(\bC)_{\sa}^{l}$ with
		\[\sup_{n}\|Y^{(n)}\|_{\infty}<\infty,\]
		\[\ell_{Y^{(n)}}\to_{n\to\infty}\ell_{y}.\]
		
		Suppose that $d\in \bN$, that $x\in M_{\sa}^{d}$ and that $Y^{(n)}$ is a microstates sequence for $y$ as above. Fix
		\[R>\max(\|x\|_{\infty},\sup_{n}\|Y^{(n)}\|_{\infty})\]
		For $n\in \bN$, and $\mathcal{O}$ a weak$^{*}$-neighborhood of $\ell_{x,y}$ we let
		\[\Gamma^{(n)}_{R}(\mathcal{O}|A^{(n)})=\{B\in \bM_{n}(\bC)_{\sa}^{r}:(B,Y^{(n)})\in \Gamma^{(n)}(\mathcal{O})\}.\]
	\end{defn}

	We let $K_{\varepsilon}(\Omega,\|\cdot\|_{2})$ be the minimal cardinality of a subset of $\bM_{n}(\bC)^{d}$ which $(\varepsilon,\norm{\cdot}_2)$-covers $\Omega$.
	Let $(M,\tau)$ be a tracial von Neumann algebra, $d,s,l\in \bN$, and let $x\in M_{\sa}^{d},y\in M_{\sa}^{s},a\in M_{\sa}^{l}$. Suppose that $(A^{(n)})_{n}$ is a microstates sequence for $a$.
	Fix $R\in (0,\infty)$ with $R>\max(\|x\|_{\infty},\|y\|_{\infty},\sup_{n}\|A^{(n)}\|_{\infty})$. For $n\in \bN$, let $\pi_{r}\colon M_{n}(\bC)_{\sa}^{d+s}\to M_{n}(\bC)_{\sa}^{d}$ be the projection onto the first $d$-coordinates.
	For a weak$^{*}$-neighborhood $\mathcal{O}$ Of $\ell_{x,y,a}$, and $n\in \bN$, we set
	\[\Gamma^{(n)}_{R}(x|A^{(n)}:\mathcal{O})=\pi_{r}(\Gamma^{(N)}_{R}(\mathcal{O}|A^{(k)})).\]
	We will heuristically refer to the collection of spaces $(\Gamma^{(n)}_{R}(x|A^{(n)}:\mathcal{O}))_{n,\mathcal{O}}$ as the microstates spaces for $x$ relative to $a$ in the presence of $y$.
	Roughly speaking, the usual microstates spaces  correspond to embeddings of $W^{*}(x,a)$ into an ultraproduct of matrices. As in Section \ref{subsec:ultraproduct} \emph{relative} microstates spaces in the presence correspond to embeddings of $W^{*}(x)$ into ultraproducts of matrices which both restrict to the embedding of $W^{*}(a)$ given by $(A^{(n)})_{n}$, and have an extension to $W^{*}(x,y,a)$.
	
	We  now define
	\begin{align*}
		\underline{h}_{R,\varepsilon}(x|(A^{(n)})_{n}:\mathcal{O})&:=\liminf_{n\to\infty}\frac{1}{n^{2}}\log K_{\varepsilon}(\Gamma_{R}^{(n)}(x|A^{(n)};\mathcal{O}),\|\cdot\|_{2}), \\
		\underline{h}_{R,\varepsilon}(x|(A^{(n)})_{n}:y) &:= \inf_{\mathcal{O}} \underline{h}_{R,\varepsilon}(x|(A^{(n)})_{n}:\mathcal{O}),
	\end{align*}
	where the infimum is over all weak$^{*}$-neighborhoods $\mathcal{O}$ of $\ell_{x,y,a}$.  If $W^{*}(a)$ is diffuse and hyperfinite, and if $W^{*}(a)\subseteq W^{*}(x)$, we then define
	\[\underline{h}(x:y) := \sup_{\varepsilon > 0} \underline{h}_{R,\varepsilon}(x|(A^{(n)})_{n}:y).
	\]
	By standard methods, it can be shown that $\underline{h}(x:y)$ does not depend upon the choice of $R$, and this justifies dropping it from the notation. If $W^{*}(a)$ is diffuse and hyperfinite, then the above quantity is independent of $a$ and the choice of $(A^{(n)})_{n}$ by the same arguments in \cite[Lemma A.5, Corollary A.6]{Hayes2018}. In \cite{Hayes2018}, this is only shown when $a$ is a single element which generates a diffuse, abelian von Neumann algebra, but the same argument works for any tuple generating a diffuse, hyperfinite von Neumann algebra.   By the same arguments as in \cite[Theorem A.9]{Hayes2018} one can show that if $d',s'\in \bN$ and $x'\in M_{\sa}^{d'},y\in M_{\sa}^{s'}$ satisfies $W^{*}(x')=W^{*}(x)$, $W^{*}(x,y)=W^{*}(x',y')$, and if $W^{*}(x),W^{*}(y)$ are diffuse, then $\underline{h}(x:y)=\underline{h}(x':y')$.
	Thus if $M_{1}\leq M_{2}\leq M$ and $M_{1},M_{2}$ are finitely generated and diffuse, we can set
	\[\underline{h}(M_{1}:M_{2})=\underline{h}(x:y),\]
	where $x_{j},j=1,2$ are finite self-adjoint tuples in $M$ with $W^{*}(x_{1})=M_{1}$, $W^{*}(x_{1},x_{2})$. We then set
	\[\underline{h}(M_{1}:M)=\inf_{Q}\underline{h}(M_{1}:Q),\]
	where the infimum is over all finitely generated subalgebras $Q$ of $M$ which contain $M_{1}$. Finally, for $N\leq M$, we let
	\[\underline{h}(N:M)=\sup_{P}\underline{h}(P:M),\]
	where the supremum is over all finitely generated, diffuse subalgebras $P$ of $N$. We call $\underline{h}(N:M)$ the \emph{lower $1$-bounded entropy of $N$ in the presence of $M$}, and we set $\underline{h}(M)=\underline{h}(M:M)$.
	The above definition using relative microstates can intuitively be thought of describing the $1$-bounded entropy as a measurement of how many embeddings $N$ has into an ultraproduct of matrices which both restrict to a given embedding of $W^{*}(a)$ and have an extension to $M$. The definition of $1$-bounded entropy using unitary conjugation orbits can be thought of as a measurement of how many embeddings $M$ has into an ultraproduct of matrices modulo unitary conjugation. It is a fact that when $N = M$, these measurements are the same \cite[Lemma A.5]{Hayes2018}.
	
	We remark that the properties of $h$ given in Section \ref{sec:preliminary} are satisfied, \emph{mutatis mutandis}, for $\underline{h}$ with the exception that the analogue of item \ref{I:subadditivity of 1 bdd ent} is that
	\[
	\underline{h}(N_{1}\vee N_{2}:M)\leq \underline{h}(N_{1}:M)+h(N_{2}:M)\textnormal{ if $N_{1}\cap N_{2}$ is diffuse}.
	\]
	The lower bound for the $1$-bounded entropy of direct sums is as follows.
	
	\begin{lemm} \label{lemm: h direct sum}
		Suppose that $J$ is a countable index set, and that $((M_{j},\tau_{j}))_{j\in J}$ are diffuse tracial von Neumann algebras. Let $(\lambda_{j})_{j\in J}$ be positive numbers with $\sum_{j}\lambda_{j}=1$ and set $(M,\tau)=\bigoplus_{j}\lambda_{j}(M_{j},\tau_{j})$.
		Suppose that $N_{j}\leq M_{j},j\in J$ are diffuse and set $N=\bigoplus_{j}N_{j}$. Then
		\[
		\underline{h}(N:M)\geq \sum_{j}\lambda_{j}^{2}\underline{h}(N_{j}:M_{j}),\]
		\[h(N:M)\geq \lambda_{i}^{2}h(N_{i}:M_{i})+\sum_{j\ne i}\lambda_{j}^{2}\underline{h}(N_{j}:M_{j})\textnormal{ for all $i\in J$.}
		\]
	\end{lemm}
	
	\begin{proof}
		If there is a $j \in J$ so that $M_{j}$ is not Connes-embeddable, then both sides are $-\infty$, and there is nothing to prove. So we will assume without further comment throughout the proof that all algebras involved are Connes-embeddable.
		The proofs for $h$ and $\underline{h}$ are nearly identical, so we will only give the proof for $\underline{h}$. First, consider the case where $J = \{1,2\}$.  Fix tracial von Neumann algebras $(M_{j},\tau_{j}),j=1,2$ and diffuse $N_{j}\leq M_{j}$. Without loss of generality, we may assume that $N_{j},M_{j}$ are finitely generated, so let $r_{j},s_{j}\in \bN,j=1,2$ and let $x_{j}\in (N_{j})_{\sa}^{r_{j}}$,$y_{j}\in (M_{j})_{\sa}^{s_{j}}$ be generating sets.
		Let $z_{j},j=1,2$ be the identity of $M_{j}$, regarded as a projection in $M_{1}\oplus M_{2}$. Let $a_{j}\in (N_{j})_{\sa},j=1,2$ be an element with diffuse spectrum. We will use microstates relative to $b=(a_{1}\oplus 0,0\oplus a_{1},z_{1},z_{2})$ to compute the entropy. So let $C^{(n)}\in \bM_{n}(\bC)_{\sa}^{4}$ be a microstates sequence for $b$. To ease notation set $A_{j}^{(n)}=C^{(n)}_{j},j=1,2$ and $Z^{(n)}_{j}=C^{(n)}_{j+2},j=1,2$. We may, and will, assume that
		\begin{itemize}
			\item $Z^{(n)}_{1},Z^{(n)}_{2}$ are diagonal projections with $Z^{(n)}_{1}+Z^{(n)}_{2}=1$ for all $n$,
			\item $n\tr_{n}(Z^{(n)}_{1})=\lfloor{\tau(z_{1})n\rfloor}$ for all $n$,
			\item $Z^{(n)}_{j}A^{(n)}_{j}Z^{(n)}_{j}=A^{(n)}_{j}$ for all $n$ and all $j=1,2,$
			\item $A^{(n)}_{j}$ are diagonal for all $n,$ and all $j=1,2$.
		\end{itemize}
		Set $b_{j}=(a_{j},z_{j}),j=1,2,$ and let $B^{(n)}_{j}=(A^{(n)}_{j},\id_{\ell_{n,j}})$ where we regard $A^{(n)}_{j}$ as an element of $M_{\ell_{n,j}}$ where $\ell_{n,j}=n\tau(Z^{(n)}_{j})$ Set $x=(x_{1}\oplus 0,0\oplus x_{2})$, and $y=(y_{1}\oplus 0,0\oplus y_{2})$. Finally, fix \[
		R>\max(\|x\|_{\infty},\|y\|_{\infty},\sup_{n}\|B^{(n)}\|_{\infty}).\]
		Given a neighborhood $\mathcal{O}$ of $\ell_{x,y,b}$ we may find neighborhoods $\mathcal{O}_{j},j=1,2$ of $\ell_{x_{j},y_{j},b_{j}}$ so that
		\[\Gamma^{(n)}_{R}(x|B^{(n)}:\mathcal{O})\supseteq \Gamma^{(\lfloor{\tau(z_{1})n\rfloor})}_{R}(x_{1}|B^{(\lfloor{\tau(z_{1})n\rfloor}))}_{1}:\mathcal{O}_{1})\oplus \Gamma^{(\lceil{\tau(z_{2})n\rceil})}_{R}(x_{2}|B^{(\lceil{\tau(z_{2})n\rceil})}_{2};\mathcal{O}_{2})\textnormal{ for all sufficiently large $n$.}\]
		Given $\varepsilon>0$, let $\Omega_{n,j}\subseteq \Gamma^{(\ell_{n,j})}(x_{j}|B^{(\ell_{n,j})}_{j};\mathcal{O}_{j})$ be  $\varepsilon$-separated with respect to $\|\cdot\|_{2}$ (meaning that distinct points in $\Omega_{n,j}$ are distance at least $\varepsilon$ from each other). Then for all sufficiently large $n$, we have that $\Omega_{n,1}\oplus \Omega_{n,2}$ is an $\varepsilon$-separated subset of $\Gamma^{(n)}(x|B^{(n)};\mathcal{O})$ and thus has cardinality at most $K_{\varepsilon/2}(\Gamma^{(n)}_{R}(x|B^{(n)};\mathcal{O}),\|\cdot\|_{2})$.
		
		Thus for all $\varepsilon>0$, and for all sufficiently large $n$,
		\[K_{\varepsilon/2}(\Gamma^{(n)}_{R}(x|B^{(n)}:\mathcal{O}),\|\cdot\|_{2})\geq K_{\varepsilon}(\Gamma^{(\lfloor{\tau(z_{1})n\rfloor})}_{R}(x_{1}|B^{(n)}_{1}:\mathcal{O}_{1}),\|\cdot\|_{2})K_{\varepsilon}( \Gamma^{(\lceil{\tau(z_{2})n\rceil})}_{R}(x_{2}|B^{(n)}_{2}:\mathcal{O}_{2}),\|\cdot\|_{2}).\]
		Since $\{\lfloor{\tau(z_{1})n\rfloor}:n\in \bN\}$ and $\{\lceil{\tau(z_{2})n\rceil}:n\in \bN\}$ have finite complement in $\bN$, and since limit infimums are super additive, the above inequality shows that for all $\varepsilon>0$ we have
		\[\underline{h}_{R,\varepsilon/2}(x|(B^{(n)})_{n}:\mathcal{O})\geq \tau(z_{1})^{2}\underline{h}_{R,\varepsilon}(x_{1}|B^{(n)}_{1}:y_{1})+\tau(z_{2})^{2}\underline{h}_{R,\varepsilon}(x_{2}|B^{(n)}_{2}:y_{2}).\]
		Letting $\mathcal{O}$ decrease to $\ell_{x,y,b}$ and letting $\varepsilon\to 0$ completes the proof of the case where $J = \{1,2\}$.
		
		The case where $J$ is finite follows by induction.  Now suppose that $J$ is infinite and assume without loss of generality that $J = \bN$.  Let $(N_{j}\leq M_{j})_{j\in J}$ be as in the statement of the lemma. For each $j\in J$, fix a diffuse abelian $A_{j}\leq N_{j}$. For $r\in \bN$, let
		\[N_{\leq r}=\bigoplus_{j=1}^{r}N_{j}\oplus \bigoplus_{j=r+1}^{\infty}A_{j}.\]
		Since $\bigoplus_{j=r+1}^{\infty}A_{j}$ has $1$-bounded entropy zero, we have by the case of finite $J$ that
		\[\underline{h}(N_{\leq r}:M)\geq \sum_{j=1}^{r}\lambda_{j}^{2}\underline{h}(N_{j}:M_{j}).\]
		Since $N_{\leq r}$ are an increasing sequence of subalgebras of $N$ and $\bigvee N_{\leq r}=N$, we have that
		\[
		\underline{h}(N:M)=\sup_{r}\underline{h}(N_{\leq r}:M)\geq \sum_{j=1}^{\infty}\lambda_{j}^{2}\underline{h}(N_{j}:M_{j}). \qedhere
		\]
	\end{proof}
	
	\subsection{Amplification and strong $1$-boundedness}\label{sec:amp it up}
	
	This section will show that $h(M^t) = t^{-2} h(M)$ for $\mathrm{II}_1$ factor $M$ (Proposition \ref{prop: amplification}), and then conclude the proof of Proposition \ref{prop: dichotomy}.  To prove Proposition \ref{prop: amplification} for $t \in (0,1)$, we consider a projection $p \in M$ with trace $t$, and analyze the $1$-bounded entropy of certain subalgebras of $M$ related to $pMp$.  At a key point, we use the fact that if $N$ is a von Neumann subalgebra of $M$, then for $h(N:M)$ to equal $h(M)$, it is sufficient that every sequence of matricial microstates for $N$ extends to a sequence of matricial microstates for $M$.  We state this property more precisely as follows.
	
	\begin{defn}
		Let $(M,\tau)$ be a tracial von Neumann algebra, and $N\leq M$. We say the inclusion $N\leq M$ has the \emph{microstates extension property} if for every free ultrafilter $\omega\in\beta\bN\setminus\bN$, for every $N_{0}\leq N$,$M_{0}\leq M$ so that $N_{0},M_{0}$ have separable predual and $N_{0}\leq M_{0}$, and for every trace-preserving embedding $\Theta\colon N_{0}\to \prod_{k\to\omega}\bM_{k}(\bC)$ into a tracial ultraproduct of matrix algebras $\bM_{k}(\bC)$, there exists a trace-preserving embedding $\widetilde{\Theta}\colon M_{0}\to \prod_{k\to\omega}\bM_{k}(\bC)$ with $\widetilde{\Theta}\big|_{M_{0}}=\Theta$.
	\end{defn}

	The following lemma shows that the microstates extension property can be formulated in terms of microstates spaces. The proof is an exercise in understanding the definitions and is left to the reader.
	
	\begin{lemm}\label{lem: alternate definition}
		Let $(M,\tau)$ be a tracial von Neumann algebra, and let $N\leq M$. Then the inclusion $N\leq M$ has the microstates extension property if and only if for every $d,s\in \bN$,  every $x\in N_{\sa}^{d},y\in M_{\sa}^{s}$, every $R>\max(\|x\|_{\infty},\|y\|_{\infty})$ and  every neighborhood $\mathcal{V}$ of $\ell_{x,y}$ there is a neighborhood $\mathcal{O}$ of $\ell_{x}$ so that
		\[\Gamma^{(n)}_{R}(\mathcal{O})\subseteq \pi_{d}(\Gamma^{(n)}_{R}(\mathcal{V}))\mbox{ for all sufficiently large $n$}\]
		where $\pi_{d}\colon \bM_{k}(\bC)_{\sa}^{d+s}\to \bM_{k}(\bC)_{\sa}^{d}$ is the projection onto the first $d$ coordinates.
	\end{lemm}

	The following lemma gives some examples and also explain the relevance to $1$-bounded entropy.
	\begin{lemm} \label{lem:MEP}  ~
		\begin{enumerate}[(i)]
			\item Suppose that $(M,\tau)$ is a tracial von Neumann algebra, that $N\leq Q\leq M$, that $N$ is diffuse, and that the inclusion $Q\leq M$ has the microstates extension property. Then $h(N:Q)=h(N:M)$ and $\underline{h}(N:Q) = \underline{h}(N:M)$. \label{item:give yourself more room}
			\item If $(M,\tau)$ is Connes embeddable, and $N\leq M$ is hyperfinite, then the inclusion $N\leq M$ has the microstates extension property.  \label{item:Connes did it}
			\item If $(M_{j},\tau_{j}),j=1,2$ are Connes-embeddable, then $M_{1}*1\leq M_{1}*M_{2}$ has the microstates extension property.  \label{item:Voiculescu did it}
		\end{enumerate}
	\end{lemm}
	
	We remark that (\ref{item:give yourself more room}) makes sense from our intuitive description of $1$-bounded entropy. The quantity $h(N:Q)$ (resp. $h(N:M)$) is supposed to be a measurement of ``how many" embeddings there are of $N$ into an ultraproduct of matrices which have an extension to $Q$ (resp. $M$). If $Q\leq M$ has the microstates extension property, then every embedding of $Q$  extends to $M$ and thus the quantities $h(N:Q)$,$h(N:M)$ should be the same.
	
	\begin{proof}
		
		(\ref{item:give yourself more room}): This is an exercise from Lemma \ref{lem: alternate definition}.
		
		(\ref{item:Connes did it}): Without loss of generality we may, and will, assume that $N,M$ have separable predual. Fix a free ultrafilter $\omega\in\beta\bN\setminus\bN$. Let $\Theta\colon N\to \prod_{k\to\omega}\bM_{k}(\bC)$ be a trace-preserving embedding. Since $M$ is Connes-embeddable, there exists a trace-preserving embedding $\Psi\colon M\to \prod_{k\to\omega}\bM_{k}(\bC)$. By \cite{Connes, JungTubularity} there exists a unitary $u\in \prod_{k\to\omega}\bM_{k}(\bC)$ so that $\Psi\big|_{N}=\ad(u)\circ \Theta$. Set $\widetilde{\Theta}=\ad(u^{*})\circ \Psi$. Then $\widetilde{\Theta}$ is the desired extension.
		
		(\ref{item:Voiculescu did it}):  Without loss of generality we may, and will, assume that $M_{1},M_{2}$ have separable predual. Our desired result is now a consequence of \cite[Theorem 2.4]{Voiculescu1998}, \cite[Corollary 0.2]{PoppArg}. Namely, fix a free ultrafilter $\omega\in\beta\bN\setminus\bN$, and let $\mathcal{M}=\prod_{k\to\omega}\bM_{k}(\bC)$. Then, by assumption, there exists trace-preserving embeddings $\Theta_{j}\colon M_{j}\to \mathcal{M}$,$j=1,2$. By \cite[Theorem 2.4]{Voiculescu1998}, \cite[Corollary 0.2]{PoppArg} there exists a unitary $u\in \mathcal{M}$ which is Haar distributed and freely independent of $\Theta_{1}(M_{1})\vee \Theta_{2}(M_{2})$. Then $\Theta_{1}(M_{1}),u\Theta_{2}(M_{2})u^{*}$ are freely independent, and this produces an extension $\widetilde{\Theta}\colon M_{1}*M_{2}\to \mathcal{M}$ of $\Theta$.
	\end{proof}
	
	Now we are ready to prove Proposition \ref{prop: amplification}, and in fact we also give the analogous result for $\underline{h}$.  We remark that the point $h(M^t) \leq t^{-2} h(M)$ for $t \in (0,1)$ was already shown in \cite[Propostion A.13(ii)]{Hayes2018}.
	
	\begin{prop}\label{prop: amplification 2}
		Suppose that $M$ is a $\textrm{II}_{1}$-factor and $\tau$ is its canonical trace. For $t\in (0,\infty)$, let $M^{t}$ be the $t^{th}$ compression of $M$. Then
		\[
		h(M^{t}) = \frac{1}{t^{2}}h(M)
		\quad \text{and} \quad \underline{h}(M^{t}) = \frac{1}{t^{2}}\underline{h}(M).
		\]
	\end{prop}
	
	\begin{proof}
		We first handle the case when $t \in (0,1)$. Fix a hyperfinite $\textrm{II}_{1}$-subfactor $R$ of $M$. Let $p\in R$ be a projection with $\tau(p)=t$. Observe that $M^{t}\cong pMp$. Let $N=pMp+(1-p)R(1-p)$
		Since $R$ is a factor, we may find partial isometries $v_{1},\cdots,v_{n}\in R$ so that $v_{j}^{*}v_{j}\leq p$ and $\sum_{j}v_{j}v_{j}^{*}=1$. Thus
		\[
		N\vee R\supseteq W^{*}(v_{1},v_{2},\cdots,v_{n})\vee (pMp+(1-p))=M.\]
		So
		\begin{equation}\label{eqn:compression splitting}
			N\vee R=M.
		\end{equation}
		\emph{Step 1. We show that $\underline{h}(M^{t})\geq \frac{1}{t^{2}}\underline{h}(M)$.}
		By (\ref{eqn:compression splitting}) and the fact that $N\cap R=pRp+(1-p)R(1-p)$ is diffuse, we have
		\[\underline{h}(M)=\underline{h}(N\vee R)\leq \underline{h}(N)+h(R)=\underline{h}(N)\leq t^{2}\underline{h}(pMp)+(1-t)^{2}h((1-p)R(1-p))=t^{2}\underline{h}(pMp)=t^{2}\underline{h}(M^{t}).\]
		Here we use the analogues of \cite[Lemma A.12 and Proposition A.13 (i)]{Hayes2018} for \underline{h}.

		\emph{Step 2. We prove that $\underline{h}(M^{t})\leq \frac{1}{t^{2}}\underline{h}(M)$.}
		We start with the following claim.
		
		\emph{Claim. $N\leq M$ has the microstates extension property.} Suppose that $\omega$ is a free ultrafilter on the natural numbers, and let $\mathcal{M}=\prod_{k\to\omega}\bM_{k}(\bC)$ be the tracial ultraproduct of $\bM_{k}(\bC)$. Let $\Theta\colon N\to \mathcal{M}$ be a trace-preserving embedding. Since $N\cap R$ is hyperfinite, we know that $N\cap R\leq R$ has the microstates extension property by Lemma \ref{lem:MEP} (\ref{item:Connes did it}).  So there exists a trace-preserving embedding $\Psi\colon R\to  \mathcal{M}$ with $\Psi\big|_{N\cap R}=\Theta\big|_{N\cap R}$. Let $v_{1},\cdots,v_{n}$ be as before the proof of Step 1. Define $\widetilde{\Theta}\colon M\to \mathcal{M}$ by
		\[\widetilde{\Theta}(x)=\sum_{i,j}\Psi(v_{i})\Theta(v_{i}^{*}xv_{j})\Psi(v_{j})^{*}.\]
		Observe that $v_{i}^{*}Mv_{j}\subseteq pMp\subseteq N$ for all $i,j$, so the formula above makes sense. For all $1\leq i,j,k,l\leq n$ and all $x,y\in M$ we have
		\begin{align*}
			\Psi(v_{i})\Theta(v_{i}^{*}xv_{j})\Psi(v_{j})^{*}\Psi(v_{k})\Theta(v_{k}^{*}yv_{l})\Psi(v_{l})^{*}&=\Psi(v_{i})\Theta(v_{i}^{*}xv_{j})\Psi(v_{j}^{*}v_{k})\Theta(v_{k}^{*}yv_{l})\Psi(v_{l})^{*}\\
			&=\delta_{j=k}\Psi(v_{i})\Theta(v_{i}^{*}xv_{j})\Theta(v_{j}^{*}v_{j})\Theta(v_{j}^{*}yv_{l})\Psi(v_{l})^{*}\\
			&=\delta_{j=k}\Psi(v_{i})\Theta(v_{i}^{*}xv_{j}v_{j}^{*}v_{j}v_{j}^{*}yv_{l})\Psi(v_{l})^{*}\\
			&=\delta_{j=k}\Psi(v_{i})\Theta(v_{i}^{*}xv_{j}v_{j}^{*}yv_{l})\Psi(v_{l})^{*}.
		\end{align*}
		From here, it is direct to show that $\widetilde{\Theta}(xy)=\widetilde{\Theta}(x)\widetilde{\Theta}(y)$ for all $x,y\in M$. It is also direct to show that $\widetilde{\Theta}$ preserves adjoints. Finally, for all $x\in M$:
		\[\tau_{\omega}(\widetilde{\Theta}(x))=\sum_{i,j}\tau_{\omega}(\Theta(v_{i}^{*}xv_{j})\Psi(v_{j}^{*}v_{i}))=\sum_{i}\tau_{\omega}(\Theta(v_{i}^{*}xv_{i}v_{i}^{*}v_{i}))=\sum_{i}\tau(xv_{i}v_{i}^{*})=\tau(x).\]
		This proves the claim.
		
		By (\ref{eqn:compression splitting}) and the fact that $N\cap R$ is diffuse, we have that
		\[
		\underline{h}(M) = \underline{h}(N\vee R:M) = \underline{h}(N:M) = \underline{h}(N),
		\]
		the last step following by the claim and Lemma \ref{lem:MEP} (\ref{item:give yourself more room}). By Lemma \ref{lemm: h direct sum},
		\[\underline{h}(M)=\underline{h}(N)\geq t^{2}\underline{h}(pMp).
		\]
		This completes the proof of Step 2 and hence also the proof for $\underline{h}$ when $t\in (0,1)$. The case $t=1$ is trivial, whereas if $t>1$, then $\frac{1}{t}\in (0,1)$ and so
		\[
		\underline{h}(M)=\underline{h}((M^{t})^{\frac{1}{t}})=t^{2}\underline{h}(M^{t}),
		\]
		For the case of $h$, it was already shown in \cite[Proposition A.13(ii)]{Hayes2018} that $h(M^{t})\leq \frac{1}{t^{2}}h(M)$ for $t \in (0,1)$, and the proof of the opposite inequality proceeds in a similar manner to Step 2 above.
	\end{proof}
	
	With Lemma \ref{lemm: h direct sum} on direct sums and Proposition \ref{prop: amplification 2} on amplifications in hand, we are ready to finish the proof of Proposition \ref{prop: dichotomy} showing that strong $1$-boundedness of all tracial von Neumann algebras with Property (T) is equivalent to $h(M) \leq 0$ for all $\mathrm{II}_1$ factors with Property (T).
	
	\begin{proof}[Proof of Proposition \ref{prop: dichotomy}]
		(iii) $\implies$ (ii).  Assume that every $\mathrm{II}_1$ Property (T) factor has nonpositive $1$-bounded entropy.  Let $M$ be a tracial von Neumann algebra with Property (T), and we will show that $h(M) \leq 0$.
		Decomposing the center of $M$ into diffuse and atomic parts, we see that
		there is a countable index set $J$ (potentially empty) so that
		\[M=M_{0}\oplus \bigoplus_{j\in J}M_{j},\]
		where $M_{0}$ either has diffuse center or is $\{0\}$, and each $M_{j}$ is a factor. Since $M_{0}$ either has diffuse center or is $\{0\}$, we know $h(M_{0})\leq 0$. Since $M$ has Property (T), each $M_{j}$ has Property (T) by \cite[Proposition 4.7.2]{PopaL2Betti}.  Thus, by \eqref{eq: direct sum upper bound}, we have $h(M) \leq 0$.
		
		(ii) $\implies$ (i).  If every tracial von Neumann algebra with Property (T) satisfies $h(M) \leq 0$, then it is strongly $1$-bounded since strong $1$-boundedness is equivalent to $h(M) < \infty$ by \cite[Proposition A.16]{Hayes2018}.
		
		(i) $\implies$ (iii).  Proceeding by contraposition, assume that $N$ is a Property (T) factor with $h(N) > 0$, and we will show that there is a tracial von Neumann algebra $(M,\tau)$ with Property (T) such that $h(M,\tau) = \infty$.  Let $(N^{4^{-k}},\tau_k)$ be the compression of $M$, and let
		\[
		(M,\tau) = \bigoplus_{k \in \bN} 2^{-k} (N^{4^{-k}},\tau_k).
		\]
		By \cite[Proposition 4.7.1]{PopaL2Betti} we know that $M$ has Property (T).  By Lemma \ref{lemm: h direct sum} and Proposition \ref{prop: amplification 2}, for each $j \in \bN$, we have
		\[
		h(M,\tau) \geq 4^{-j} h(N^{4^{-j}}) + \sum_{k \neq j} 4^{-k} \underline{h}(N^{4^{-j}}) = 4^j h(N) + \sum_{k \neq j} 4^k \underline{h}(N).
		\]
		Since $h(N)>0$, we know that $N$ is Connes-embeddable, and thus $\underline{h}(N)\geq 0$. So
		\[
		h(M)\geq 4^j h(N) \textnormal{ for all $j\in \bN$.}
		\]
		Letting $j\to\infty$ we see that $h(M)=\infty$, i.e. $M$ is not strongly $1$-bounded.
		
		Finally, we show that for each finite von Neumann algebra $M$ with Property (T), there exists a faithful normal tracial state $\tau$ such that $(M,\tau)$ is strongly $1$-bounded.  As in (iii) $\implies$ (ii), write $M = \bigoplus_{j=0}^{\infty} M_j$ such that $M_0$ is zero or has diffuse center, and $M_j$ is a factor for $j \geq 1$.
		Let $\tau_j$ be the unique tracial state on $M_j$.  Since $h(M_0) \leq 0$ and since $h(M_j) < \infty$ for $j \geq 1$ by Theorem \ref{thm: Property T}, we may choose nonnegative constants $(\lambda_j)_{j=0}^{\infty}$ such that $\sum_{j=0}^{\infty} \lambda_j = 1$, $\sum_{j \in \bN_0} \lambda_j^2 h(M_j) < \infty$, and $\lambda_{j}>0$ if and only if $M_{j}\ne 0$.  Let $\tau$ be the faithful normal tracial state on $M$ given by $\bigoplus_{j=0}^{\infty} \lambda_j \tau_j$.  It follows by \eqref{eq: direct sum upper bound} that $h(M,\tau) < \infty$.
	\end{proof}

	
	\subsection{Direct sums and free entropy dimension}\label{sec:FED d sums}
	
	In this section, we show that Theorem \ref{thm: Property T} implies Jung and Shlyakhtenko's result that a Property (T) tracial von Neumann algebra has microstates free entropy dimension at most $1$ with respect to every finite generating tuple \cite{JungS}.  In fact, at the end of the section, we also sketch how to generalize the argument to show that $\delta_{0}(x)\leq 1$ for any infinite generating tuple $x$.  First, we recall the definition of Voiculescu's microstates free entropy dimension.
	
	
	\begin{defn}
		Let $(M,\tau)$ be a tracial von Neumann algebra, $x\in M^{d}_{\sa}$. Fix $R>\|x\|_{\infty}$. For $\varepsilon>0$, and a weak$^{*}$-neighborhood $\mathcal{O}$ of $\ell_{x}$, we set
		\[\delta_{R,\varepsilon}(\mathcal{O})=\limsup_{k\to\infty}\frac{\log K_{\varepsilon}(\Gamma^{(k)}_{R}(\mathcal{O}),\|\cdot\|_{2})}{k^{2}|\log(\varepsilon)|},\]
		\[\delta_{R,\varepsilon}(x)=\inf_{\mathcal{O}}\delta_{\varepsilon}(\mathcal{O}),\]
		where the infimum is over all weak$^{*}$-neighborhoods $\mathcal{O}$ of $\ell_{x}$. We then set
		\[\delta_{0}(x)=\limsup_{\varepsilon\to 0}\delta_{R,\varepsilon}(x).\]
		We call $\delta_{0}(x)$ the \emph{microstates free entropy dimension of $x$}.
	\end{defn}
	By standard methods, $\delta_{0}(x)$ does not depend upon the choice of $R$ and this justifies dropping it from the notation.
	This is not the original definition in \cite{Voiculescu1996}, however by \cite[Corollary 2.4]{JungLemma} they are the same.  The following lemma, based on previous work of Jung, describes the behavior free entropy dimension under direct sums.
	
	\begin{lemm}\label{lemm:approximating MFED by a large central corner}
		Let $(M,\tau)$ be a tracial von Neumann algebra, $d\in \bN$, and $x\in M^{d}_{\sa}$. Let $J$ be a countable set and $(z_{j})_{j\in J}$ be central projections in $M$ with $\sum_{j\in J}z_{j}=1$.
		For $j\in J$, let $x_{z_{j}}$ be $xz_{j}$ regarded as an element of $Mz_{j}$.  Then
		\[
		\delta_{0}(x)-1\leq \sum_{j\in J}\tau(z_{j})^{2}(\delta_{0}(x_{z_{j}})-1).
		\]
	\end{lemm}
	
	\begin{proof}
		Fix $R>\|x\|_{\infty}$. We first handle the case where $J$ is finite. By induction, to handle the case of finite $J$ it suffices to handle the case where $J=\{1,2\}$. In this case we use $z$ for $z_{1}$.
		Let $P_{n}\in \bM_{n}(\bC)_{\sa}$ be microstates for $z$ such that each $P_{n}$ is an orthogonal projection.
		By \cite[Lemma 3.2 and Corollary 4.3]{JungHyperFiniteIneq}, we have
		\[\delta_{0}(x)=\delta_{0}(x,z)=\delta_{0}(z)+\limsup_{\varepsilon\to 0}\frac{h_{R,\varepsilon}(x|(P_{k}))}{\log(1/\varepsilon)}=2\tau(z)(1-\tau(z))+\limsup_{\varepsilon\to 0}\frac{h_{R,\varepsilon}(x|(P_{k}))}{\log(1/\varepsilon)},\]
		where in the last step we use \cite[Corollary 5.8]{JungRegularity}. It follows from the proof of \cite[Proposition A.13(i)]{Hayes2018} that
		\[h_{R,6\varepsilon}(x|(P_{k})_{k})\leq \tau(z)^{2}h_{R,\varepsilon}(x_{z})+(1-\tau(z))^{2}h_{R,\varepsilon}(x_{1-z}), \textnormal{ for all sufficiently small $\varepsilon$.}\]
		Dividing by $\log(1/\varepsilon)$ and letting $\varepsilon\to 0$ we obtain that
		\[\delta_{0}(x)\leq 2 \tau(z)(1-\tau(z))+\tau(z)^{2}\delta_{0}(x_{z})+(1-\tau(z))^{2}\delta_{0}(x_{1-z}),\]
		and by direct computation this is equivalent to the desired inequality.
		
		We now handle the case of infinite $J$. We may, and will, assume that $J=\bN$. For $n\in \bN$,  let $z_{\leq n}=\sum_{j=1}^{n}z_{j}$. Then, by the case of finite $J$:
		\[\delta_{0}(x)-1\leq (1-\tau(z_{\leq n}))^{2}(\delta_{0}(x_{1-z_{\leq n}})-1)+\sum_{j=1}^{n}\tau(z_{j})^{2}\delta_{0}(x_{z_{j}}).\]
		We have that $\delta_{0}(x_{1-z_{\leq n}})\leq d$, and thus
		\[\delta_{0}(x)-1\leq (1-\tau(z_{\leq n}))^{2}(d-1)+\sum_{j=1}^{n}\tau(z_{j})^{2}\delta_{0}(x_{z_{j}}).\]
		The proof is thus completed by letting $n\to \infty.$

	\end{proof}
	
	Since being strongly $1$-bounded implies microstates free entropy dimension at most $1$ with respect to any set of generators, the preceding lemma automatically implies the following.
	
	\begin{cor}\label{cor:direct sum of S1B things}
		Let $J$ be a countable set and $((M_{j},\tau_{j}))_{j\in J}$ tracial von Neumann algebras. Suppose that $(\lambda_{j})_{j\in J}\in (0,1]^{J}$ with $\sum_{j\in J}\lambda_{j}=1$.
		Let $(M,\tau)=\bigoplus_{j}\lambda_{j}(M_{j},\tau_{j}).$
		Assume each $(M_{j},\tau_{j})$ is strongly $1$-bounded. Then for any $x\in M_{\sa}^{d}$ with $W^{*}(x)=M$ we have $\delta_{0}(x)\leq 1$.
	\end{cor}
	
	
	We now recover the results of Jung-Shlyakhtenko.
	\begin{cor}\label{cor:JS reproof}
		Let $(M,\tau)$ be a Property (T) von Neumann algebra which is finitely generated. Suppose that $x\in M_{\sa}^d$ satisfies $W^{*}(x)=M$. Then $\delta_{0}(x)\leq 1$.
	\end{cor}
	
	\begin{proof}
		We may write $M=M_{0}\oplus \bigoplus_{j\in J}M_{j}$ where $J$ is a (potentially empty) countable set, each $M_{j}$ is a Property (T) factor, and $M_{0}$ is either $\{0\}$ or an algebra with diffuse center. Since $M_{0}$ either has diffuse center or is $\{0\}$, we have $h(M_{0})\leq 0$. If $J$ is finite, we see that $M$ is strongly $1$-bounded by Theorem \ref{thm: Property T}, and by how $1$-bounded entropy behaves under direct sums. If $J$ is infinite, then we may apply Corollary \ref{cor:direct sum of S1B things} to complete the proof.
		
	\end{proof}
	
	We remark that if one works carefully with free entropy dimension in the presence in Lemma \ref{lemm:approximating MFED by a large central corner}, then a proof of the Corollary \ref{cor:direct sum of S1B things} can be given for infinite tuples as well. Using this, one can show that if $M$ is a Property (T) algebra, then $\delta_{0}(x)\leq 1$ for every tuple which generates $M$, even if $x$ is infinite. We will not give the full proof here, but sketch the details for the interested reader.
	
	Given a tracial von Neumann algebra $(M,\tau)$ and finite tuples $x\in M_{\sa}^{d},y\in M_{\sa}^{s}$ with $d,s\in \bN$, an $R>\max(\|x\|_{\infty},\|y\|_{\infty})$, a weak$^{*}$-neighborhood $\mathcal{O}$ of $\ell_{x,y}$, and $\varepsilon>0$ we set
	\[\delta_{R,\varepsilon}(x:\mathcal{O})=\limsup_{n\to\infty}\frac{K_{\varepsilon}(\Gamma_{R}^{(n)}(x:\mathcal{O}),\|\cdot\|_{2})}{\log(1/\varepsilon)},\]\[\delta_{R,\varepsilon}(x:y)=\inf_{\mathcal{O}}\delta_{R,\varepsilon}(x:\mathcal{O}),\]
	where the infimum is over all weak$^{*}$-neighborhoods $\mathcal{O}$ of $\ell_{x,y}$. We then define the \emph{free microstates free entropy dimension of $x$ in the presence of $y$} by
	\[\delta_{0}(x:y)=\limsup_{\varepsilon\to 0}\delta_{R,\varepsilon}(x:y),\]
	by standard methods we can show that this is independent of $R$, and this justifies dropping $R$ from the notation. One can show that if $a\in M_{\sa}^{t}$ for some $t\in \bN$ and $W^{*}(x,y)=W^{*}(x,a)$, then $\delta_{0}(x:y)=\delta_{0}(x:a)$. So we set
	\[\delta_{0}(x:W^{*}(x,y))=\delta_{0}(x:y),\]
	and this does not depend upon the choice of $y$.
	
	Now suppose that $x=(x_{j})_{j\in J}\in M_{\sa}^{J}$ for some set $J$. For a finite $F\subseteq J$, let $x_{F}\in M_{\sa}^{F}$ be given by $x_{F}=(x_{j})_{j\in F}$. We then set
	\[\delta_{0}(x_{F}:M)=\inf_{Q}\delta_{0}(x_{F}:Q),\]
	\[\delta_{0}(x:M)=\sup_{F}\delta_{0}(x_{F}:M),\]
	where the infimum is over all finitely generated $Q\leq M$ with $x_{F}\in Q^{F}_{\sa}$, and the supremum is over all finite subsets $F$ of $J$. We set $\delta_{0}(x)=\delta_{0}(x:W^{*}(x))$.
	
	To generalize Corollary \ref{cor:direct sum of S1B things} to infinite tuples, one first proves a modification of Lemma \ref{lemm:approximating MFED by a large central corner}. Namely if $x$ is a self-adjoint tuple in $M$ and $(z_{j})_{j\in J}$ are projections in $Z(M)\cap W^{*}(x)$, then
	\begin{equation}\label{eqn:ds ineq for infinite tuples}
		\delta_{0}(x:M)-1\leq \sum_{j}\tau(z_{j})^{2}(\delta_{0}(x_{z_{j}}:Mz_{j})-1)
	\end{equation}
	To prove (\ref{eqn:ds ineq for infinite tuples}) one first handles the case $x$ is a finite tuple, and $J=\{1,2\}$. The proof of (\ref{eqn:ds ineq for infinite tuples}) in this case is a minor modification of Lemma \ref{lemm:approximating MFED by a large central corner}. Namely, one  modifies the proof of \cite[Lemma 3.2 and Corollary 4.3]{JungHyperFiniteIneq}   to show that if $z\in Z(M)\cap W^{*}(x)$ is a central projection, then
	\[\delta_{0}(x,z:M)=2\tau(z)(1-\tau(z))+\limsup_{\varepsilon\to 0}\frac{h_{R,\varepsilon}(x:M)}{\log(1/\varepsilon)}
	.\]
	One then modifies the proof of \cite[Corollary 5.8]{JungRegularity} to say that $\delta_{0}(x:M)=\delta_{0}(x,z:M)$. After changing the results in \cite{JungHyperFiniteIneq, JungRegularity} to work for free entropy dimension in the presence, the proof of (\ref{eqn:ds ineq for infinite tuples}) in the case that $x$ is a finite tuple and $J=\{1,2\}$ proceeds exactly as in Lemma \ref{lemm:approximating MFED by a large central corner}. The proof of the general case of (\ref{eqn:ds ineq for infinite tuples}) from this special case also follows precisely as in Lemma \ref{lemm:approximating MFED by a large central corner}.
	
	The inequality (\ref{eqn:ds ineq for infinite tuples}) automatically shows that if $M$ is a direct sum of strongly $1$-bounded algebras and if $x$ is any self-adjoint tuple in $M$, then $\delta_{0}(x)\leq 1$, and from this one deduces a version of Corollary \ref{cor:JS reproof} where $x$ is any self-adjoint tuple (finite or not).

	\bibliographystyle{amsalpha}
	\bibliography{InnerAmen.bib}

\end{document}